\documentclass[reqno]{amsart}
\usepackage[left=1in, right=1in, top=1in]{geometry} 
\usepackage{mathrsfs}
\usepackage{amsfonts}
\usepackage{comment}
\usepackage{cases}
\usepackage{latexsym}
\usepackage{amsmath}
\usepackage[all]{xy}
\usepackage{stmaryrd}
\usepackage{amsfonts}
\usepackage{amsmath,amssymb,amscd,bbm,amsthm,mathrsfs,dsfont}

\usepackage{tikz}

\usepackage{pgflibraryarrows}
\usepackage{pgflibrarysnakes}

\usepackage[numbers,sort&compress]{natbib}
\usepackage[pagebackref]{hyperref}
\usepackage{hypernat}
\usepackage{color, hyperref}
\definecolor{blue}{rgb}{1.0,0.0,0.9}
\hypersetup{colorlinks, breaklinks,
            linkcolor=blue,urlcolor=blue,
            anchorcolor=blue, citecolor=blue}

\usepackage{fancyhdr}
\usepackage{amsxtra,ifthen}
\usepackage{verbatim}

\numberwithin{equation}{section}

\theoremstyle{plain}
\newtheorem{theorem}{Theorem}[section]
\newtheorem{lemma}[theorem]{Lemma}
\newtheorem{proposition}[theorem]{Proposition}

\theoremstyle{definition}

\newtheorem{claim}[theorem]{Claim}
\newtheorem*{thm}{Theorem}

\allowdisplaybreaks[4]
\begin{document}

\title[Non-existence of normal elements in the Iwasawa Algebra of Chevalley groups]
{Non-existence of non-trivial normal elements in the Iwasawa Algebra of Chevalley groups}

\author{Dong Han, Jishnu Ray and Feng Wei}

\address{Han: School of Mathematics and Information Science, Henan Polytechnic University, Jiaozuo 454000, P. R. China}

\email{lishe@hpu.edu.cn}

\address{Ray: Department of Mathematics, The University of British Columbia, 1984 Mathematics Road,V6T 1Z2, Vancourer, BC, Canada}

\email{jishnuray@math.ubc.ca} \email{jishnuray1992@gmail.com}

\address{Wei: School of Mathematics and Statistics, Beijing Institute of Technology,
Beijing 100081, P. R. China}

\email{daoshuo@hotmail.com}\email{daoshuo@bit.edu.cn}

\begin{abstract}
For a prime $p>2$,
let $G$ be a semi-simple, simply connected, split Chevalley group over $\mathbb{Z}_p$, 
$G(1)$ be the first congruence kernel of $G$ and $\Omega_{G(1)}$ be the mod-$p$ Iwasawa algebra 
defined over the finite field $\mathbb{F}_p$. Ardakov, Wei, Zhang \cite{ArdakovWeiZhang1} have shown 
that if $p$ is a ``nice prime " ($p \geq 5$ and $p \nmid n+1$ if the Lie algebra of $G(1)$ is of type $A_n$), 
then every non-zero normal element in $\Omega_{G(1)}$  is a unit.  Furthermore, they conjecture 
in their paper that their nice prime condition is superfluous.  
The main goal of this article is to provide an entirely new proof of Ardakov, Wei and Zhang's 
result using explicit presentation of Iwasawa algebra developed by the second author of this 
article and thus eliminating the nice prime condition,  therefore proving their conjecture.

\end{abstract}



\date{\today}
\thanks{The second author would like to thank PIMS-CNRS and the University of British Columbia for postdoctoral research grant. 
He is also thankful to Beijing Institute of Technology for its hospitality during a visit on August 2018 when this collaboration took place.}

\maketitle

\tableofcontents

\section{Introduction}
\label{xxsec1}

Let $p$ be a prime integer, and let $\mathbb{Z}_p$ denote the ring of
$p$-adic integers. A group $G$ is \textit{compact p-adic analytic}
if it is a topological group which has the structure of a $p$-adic
analytic manifold - that is, it has an atlas of open subsets of
$\mathbb{Z}^n_p$ , for some $n\geq 0$. There is a more intrinsic way to
characterize such kinds of groups.  A topological group $G$ is
compact $p$-adic analytic if and only if $G$ is a closed subgroup of
the general linear group ${\rm GL}_n(\mathbb{Z}_p)$ for some $n\geq
1$. The research object of this article is the so-called \textit{Iwasawa algebras} of $G$:
$$
\Lambda_G :=\varprojlim_{N\unlhd G} \mathbb{Z}_p[G/N],
$$
where the inverse limit is taken over the open normal subgroups $N$
of $G$. Modulo $p$, the epimorphic image of $\Lambda_G $ is denoted by 
$\Omega_G =\Lambda_G\bigotimes_{\mathbb{Z}_p}\mathbb{F}_p$ and
$$
\Omega_G : =\varprojlim_{N\unlhd G} \mathbb{F}_p[G/N],
$$
where $\mathbb{F}_p$ is the finite field of $p$ elements. 
For any odd prime $p$, Clozel in his paper \cite{Clozel1} gave explicit presentations for
the afore-mentioned two Iwasawa algebras over
the first congruence subgroup of ${\rm SL}_2(\mathbb{Z}_p)$, which is
$\Gamma_1({\rm SL}_2(\mathbb{Z}_p))={\rm Ker}({\rm SL}_2(\mathbb{Z}_p)\longrightarrow {\rm SL}_2(\mathbb{F}_p))$.
Ray \cite{Ray1, Ray2, Ray3, Ray4}  generalized Clozel's work to the following three cases: 
the first congruence kernel of a semi-simple, simply connected Chevalley group over 
$\mathbb{Z}_p$, general uniform pro-$p$ groups, and pro-$p$ Iwahori subgroups of ${\rm GL}_n(\mathbb{Z}_p)$.

 In this article, we are going 
to look at the \textit{normal elements}   in $\Omega_G$ and   the ideals generated by them. They 
are defined as $r\in \Omega_G$ such that $r\Omega_G=\Omega_G r$. Each normal element $r$ gives rise to a two-sided reflexive ideal $r\Omega_G$. Let us first recall the definition of reflexive ideals. Let $A$ be any algebra and $M$ be a left
$A$-module. We call $M$ \textit{reflexive} if the canonical mapping
$$
M\longrightarrow {\rm Hom}_{A^{\sf op}}({\rm Hom}_{A}(M,A),A)
$$
is an isomorphism. A reflexive right $A$-module is defined
similarly. We will call a two-sided ideal $I$ of $A$ \textit{reflexive}
if it is reflexive as a right and as a left $A$-module. 

Let $\Phi$ be a root system of $G$ having fixed a split maximal torus, so that the Dynkin diagram of any indecomposable component 
of $\Phi$ belongs to
$$
\{A_n(n\geq1), \, B_n(n\geq 2),\, C_n(n\geq 3),\,  D_n(n\geq 4), \, E_6, \, E_7,\,  E_8, \, F_4, \, G_2\}
$$
Let $\Phi(\mathbb{Z}_p)$ denote the $\mathbb{Z}_p$-Lie algebra constructed by using a 
Chevalley basis associated to $\Phi$.

We say that $p$ is a \textit{nice prime} for $\Phi$ if $p\geq 5$ and if $p\nmid n+1$ when $\Phi$ has 
an indecomposable component of type $A_n$. Ardakov, the third author of the current article and 
Zhang have showed that

\begin{theorem}{\rm \cite[Theorem A and Theorem B]{ArdakovWeiZhang1},\cite[Corollary 0.3]{ArdakovWeiZhang2}}\label{xxsec7.2}
Let $G$ be a torsionfree compact $p$-adic analytic group whose $\mathbb{Q}_p$-Lie algebra $\mathcal{L}(G)$ 
is split semisimple over $\mathbb{Q}_p$. Suppose that $p$ is a nice prime for the root system $\Phi$ of $\mathcal{L}(G)$. 
Then the mod-$p$ Iwasawa algebra $\Omega_G$ has no non-trivial two-sided reflexive ideals. In particular,  
every non-zero normal element of $\Omega_G$ is a unit. 
\end{theorem}

Ardakov, the third author of the current article and Zhang conjecture (see paragraph before 
section 0.4 of \cite{ArdakovWeiZhang2}) that the nice prime condition in their theorem above is superfluous. 
Now, for the rest of this paper, we assume that $G$ is a semi-simple, simply connected, split 
Chevalley group over  $\mathbb{Z}_p$, and that $G(1)$ is the first congruence kernel defined by
$$G(1):={ \rm Ker}(G(\mathbb{Z}_p)\longrightarrow G(\mathbb{Z}_p/p\mathbb{Z}_p)).$$
In this paper we prove that ``nice prime" condition in indeed superfluous for $G(1)$ and thereby confirming to Ardakov, 
the third author of the current article and Zhang's conjecture. Furthermore, our method of proof is completely 
different from that of Ardakov and in based on an explicit presentation of Iwasawa algebras and works for any prime $p>2$. 
The main theorem of this article, extending earlier works of \cite{HanWei1, HanWei2, WeiBian1, WeiBian2} for 
${\rm SL}_2(\mathbb{Z}_p)$, ${\rm SL}_3(\mathbb{Z}_p)$ and ${\rm SL}_n(\mathbb{Z}_p)$, is the following.

\begin{thm}(Theorem \ref{xxsec5.1}) 
Let $p$ be a prime with $p\geq 3$ and $G$ be a semi-simple, simply connected, split Chevalley group 
over $\mathbb{Z}_p$. Suppose that $G$ is one of the following Chevalley groups of Lie type: 
$A_{\ell}\, (\ell\geq 1), B_{\ell}\,(\ell\geq 2), C_{\ell}\,(\ell\geq 2), D_{\ell}\,(\ell\geq 3), E_6, E_7, E_8, F_4, G_2$.
Then for any nonzero element $W\in \Omega_{G(1)}$, $W$ is a normal element if and only if 
$W$, as a noncommuttaive formal power series, contains constant terms. In this case, $W$ is a unit. 
\end{thm}
A direct consequence of this result is reproving that the center of $\Omega_{G(1)}$ is trivial (again originally done by Ardakov in \cite{Ardakov}.) This is carried out in Proposition \ref{xxsec7.1}.

It should be remarked that we need $p>2$, otherwise $G(1)$ is not torsion free and 
$\Omega_{G(1)}$ is  not an integral domain and our proof is heavily based on Lazard's theory of $p$-valued group and Iwasawa algebras where we need our Iwasawa algebras to be an integral domain.

The roadmap of this article is as follows. After Introduction, we recall the  basic 
notions of Lazard basis for $p$-adic analytic groups in Section \ref{xxsec2}. 
Section \ref{xxsec3} is contributed to the computations of the lowest degree terms of the 
commutators between the generators of $\Omega_{G(1)}$\,(see Proposition \ref{xxsec3.1}). 
In Section \ref{xxsec4}, we calculate the partial differential  equations (cf. Lemma \ref{xxsec4.1}). 
The proof of our main result (Theorem \ref{xxsec5.1}) is given in Section \ref{xxsec5}.
While proving our main theorem, we need an important claim (see Claim \ref{xxsec5.3}) 
which we prove later in Section \ref{xxsec6} using Dynkin diagrams and the partial differential 
equations.  Section \ref{xxsec7} contains applications to center reproving Ardakov's result. Future questions are discussed in Section \ref{applications nonuniform}.

\section{Basic Setup on Lazard Ordered Basis}
\label{xxsec2}

Let $G$ be a semi-simple, simply connected, split Chevalley group over $\mathbb{Z}_p$, $\Phi$ be 
the root system with respect to a maximal torus, $\Pi$ be a  set of simple roots. We can view $G$ as a 
group scheme over $\mathbb{Z}$\,(cf. \cite[XXV]{GillePolo}) and the congruence kernels can be defined as
$$G(k):={ \rm ker}(G(\mathbb{Z}_p)\longrightarrow G(\mathbb{Z}_p/p^k\mathbb{Z}_p)).$$
Let 
$$
\{H_{\delta_1},\cdots, H_{\delta_{\ell}},X_{\alpha}, \alpha\in\Phi,\delta_1,\cdots,\delta_{\ell}\in\Pi\}
$$ 
be a Chevalley basis in the Lie algebra Lie$(G)$ of $G$\,(cf. \cite[Page 6]{Steinberg}). Let $x_{\alpha}(t)=\exp\, tX_{\alpha}$, 
which is the one dimensional unipotent subgroup in the $\mathbb{Q}_p$ points of $G$ if $t\in\mathbb{Q}_p$. For $t\in \mathbb{Q}_p^*$, 
we define
$$h_{\alpha}(t):=\omega_{\alpha}(t){\omega_{\alpha}(1)}^{-1},$$
where
$\omega_{\alpha}(t):=x_{\alpha}(t)x_{-\alpha}(-t^{-1})x_{\alpha}(t)$\,(cf. \cite[Lemma 19 of Section 3]{Steinberg}). 
By invoking \cite[Page 171]{Schneider}, we know that the function
\begin{align*}
\omega:\ G(1)& \longrightarrow R_+^*\cup {\infty},\\
x & \longmapsto k \ \text{for}\  x\in G(k)\backslash G(k+1)
\end{align*}
is a $p$-valuation on $G(1)$ in the sense of Lazard \cite[III, 2.1.2]{Lazard}. Here, by convention, 
we put $\omega(1)=\infty$. We recall, from \cite[Theorem 2.2]{Ray1}, that the
elements
\begin{equation}\{x_{\beta}(p),h_{\delta}(1+p), x_{\alpha}(p), \beta\in\Phi^-, \delta\in\Pi, \alpha\in\Phi^+ \}\end{equation}
form a Lazard ordered basis\,(cf.  \cite[III, 2.2.4]{{Lazard}}) for ($G(1), \omega$). The ordering in (2.1) is given 
by a fixed order on the roots such that the height function on the roots increases. Also, in (2.1), 
$\Phi^-$ (resp. $\Phi^+$) denotes the negative (resp. positive) roots in $\Phi$. Moreover, since 
$\omega(x_{\beta}(p))=\omega(x_{\alpha}(p))=\omega(h_{\delta}(1+p))$ and $p$ is assumed to 
be $>2$, $G(1)$ is also $p$-saturated with respect to its $p$-valuation \,(cf.  \cite[III, 2.2.7.1]{{Lazard}}). 
Furthermore, we would like to point out that $G(1)$  is a uniform pro-$p$ group in the sense of 
\cite{DixonduSautoyMannSegal}. A nice exposition can be found in
Sections 2, 3 and 4 of the Arxiv version\,(\cite[Arxiv]{{Ray1}})  of  \cite{Ray1}.

Now, let $g_1,\cdots, g_d\, (d=|\Phi|+|\Pi|)$ be the ordered basis as in (2.1). Taking into account 
the definition of Lazard ordered basis, we have a homomorphism
\begin{align*}
c: \mathbb{Z}_p^d & \longrightarrow G(1),\\
(y_1,\cdots,y_d)& \longmapsto g_1^{y_1}\cdots g_d^{y_d}.
\end{align*}
Let $C(G(1))$ be the set of continuous functions from $G(1)$ to $\mathbb{Z}_p$. The mapping $c$ induces, 
by pulling back functions, an isomorphism of $\mathbb{Z}_p$-modules
$$c^*: C(G(1))\simeq C(\mathbb{Z}_p^d).$$
Dualizing this isomorphism, we obtain
$$
\begin{aligned}
\Lambda_{G(1)}={\rm Hom}_{\mathbb{Z}_p}(C(G(1)), \mathbb{Z}_p) &\simeq \mathbb{Z}_p[[y_{\alpha}, y_{\delta}, \alpha\in\Phi, \delta\in\Pi]],\\
x_{\alpha}(p)-1 & \longmapsto y_{\alpha},\\
h_{\delta}(1+p)-1 & \longmapsto y_{\delta}.
\end{aligned}
$$
Identifying $y_{\alpha}$ with $x_{\alpha}(p)-1$ and $y_{\delta}$ with $h_{\delta}(1+p)-1$, we observe
that any element in $\Lambda_{G(1)}$\,(and $\Omega_{G(1)}$) can be written as a uniquely 
determined power series in $y_{\alpha}$'s and $y_{\delta}$'s with coefficients in 
$\mathbb{Z}_p$\,(resp. $\mathbb{F}_p$). Note that the ordering of $y_{\alpha}$ and $y_{\delta}$
are according to increasing height function on the roots as in the case of Lazard ordered 
basis (cf. $ y_{\alpha}= x_{\alpha}(p)-1$, $y_{\delta}=h_{\delta}(1+p)-1$, see
equation (2.1) and the discussion following it). We refer the reader to 
\cite{Schneider} for a nice introduction to Iwasawa algebras and Lazard ordered basis.

\section{Calculation of the Lowest Degree Terms of Commutators} 
\label{xxsec3}

Let $ y_{\alpha}= x_{\alpha}(p)-1$ and $y_{\delta}=h_{\delta}(1+p)-1$. In this section 
we are going to find out the lowest degree terms of all commutator relation between $y_{\alpha}$
and $y_{\delta}$ in $\Omega_{G(1)}$. The strategy is to use firstly Steinberg's Chevalley 
relation in $G(1)$ following \cite{Steinberg}\,(which gave us the explicit relations for the 
Iwasawa algebra in \cite[Lemma 3.1]{Ray1}) and then determine the commutators of 
elements $y_{\alpha}$'s and $y_{\delta}$'s. Furthermore  we determine the lowest degree
terms of the commutators (cf. Proposition \ref{xxsec3.1}).

Throughout this section $\alpha\in\Phi$, $\delta\in\Pi$, $\Pi=\{\delta_1,\cdots,\delta_{\ell}\}$, $\ell=|\Pi|$, 
$ y_{\alpha}= x_{\alpha}(p)-1$ and $y_{\delta}=h_{\delta}(1+p)-1$. We use $[-,-]$ to 
denote the commutators of elements $y_{\alpha}$'s and $y_{\delta}$'s and use $[-,-]_\circ$ to denote 
the lowest degree terms of the commutators $[-, -]$.  We recall that we are working in the mod-$p$ 
Iwasawa algebra $\Omega_{G(1)}$.

\begin{proposition} \label{xxsec3.1}
Let $r, s$ be non-negative integers and $p$ be a prime integer with $p>5$.  Then we have
\begin{enumerate}
\item[(1) ]  $[y^{p^r}_{\alpha_1}, y^{p^s}_{\alpha_2}]_{\circ}  = 0$, whenever $\alpha_1, \alpha_2\in\Phi$, $\alpha_1\neq -\alpha_2$, $\alpha_1+ \alpha_2\notin   \Phi$,
\item[(2) ]  $ [y^{p^r}_{\alpha_1}, y^{p^s}_{\alpha_2}]_{\circ}  = c_{11}y^{p^{r+s+1}}_{\alpha_1+\alpha_2}$, whenever $\alpha_1+ \alpha_2\in   \Phi$, $\alpha_1\neq -\alpha_2$,
\item[(3) ]  $ [y^{p^r}_{\alpha}, y^{p^s}_{\delta}]_{\circ}  = -\langle\alpha,\delta \rangle y^{p^{r+s+1}}_{\alpha}$,
\item[(4) ]  $ [y^{p^r}_{\alpha}, y^{p^s}_{-\alpha}]_{\circ}  = -\sum_{i=1}^{\ell}n_iy^{r+s+1}_{\delta_{i}}$,
 \end{enumerate}
where $c_{11}\in\{\pm1,\pm2,\pm3\}$, 
$\langle\alpha,\delta \rangle=\frac{2(\alpha\,|\, \delta)}{(\delta\,|\, \delta)}\in \mathbb{Z}$ is 
the so-called Cartan integer\,{\rm (}cf. {\rm \cite[Page 2]{Steinberg}}{\rm )}. The $n_i$'s in $(4)$ are such that 
$\alpha=\sum_{i=1}^{\ell}n_i\delta_i$, $\delta_i\in\Pi$. By the discussion in {\rm \cite[Page 5]{Steinberg}}, 
we know that $\langle\alpha,\delta \rangle\in$ $\{\pm1, 0,\pm1,\pm2\}$.
\end{proposition}

\begin{proof}
If $\alpha_1, \alpha_2\in\Phi$, $\alpha_1\neq -\alpha_2$, $\alpha_1+ \alpha_2\notin   \Phi$, 
then Example (a) in  \cite[Page 24]{Steinberg} gives
$$
x_{\alpha_1}(p^{r+1})x_{\alpha_2}(p^{s+1})=x_{\alpha_2}(p^{s+1})x_{\alpha_1}(p^{r+1}).
$$
Note that $y^{p^r}_{\alpha_1}=x_{\alpha_1}(p^{r+1})-1$\, (we are working in $ \Omega_{G(1)}$, which is defined over $\mathbb{F}_p$), 
we therefore get
$$ [y^{p^r}_{\alpha_1}, y^{p^s}_{\alpha_2}]= [y^{p^r}_{\alpha_1}, y^{p^s}_{\alpha_2}]_{\circ}=0,
$$
which is the desired assertion (1).

If $\alpha_1+ \alpha_2\in   \Phi$ and  $\alpha_1\neq -\alpha_2$, then applying the relation R2 of \cite[Page 30]{Steinberg} yields
$$x^{p^{r+1}}_{\alpha_1}x^{p^{s+1}}_{\alpha_2}=\prod_{i,j>0}x_{i\alpha_1+j\alpha_2}(c_{ij}p^{i(r+1)+j(s+1)})
\times x_{\alpha_2}(p^{s+1})x_{\alpha_1}(p^{r+1}).
$$
By Page 64 of \cite[Expos\'e  XXIII]{GillePolo}, we see that $c_{ij}\in\{\pm 1, \pm2, \pm3\}$. 
Thus we arrive at
\begin{align*}
(1+y^{p^r}_{\alpha_1})(1+y^{p^s}_{\alpha_2})=&(1+y^{p^{r+s+1}}_{\alpha_1+\alpha_2})^{c_{11}}\prod_{i,j>0, (i,j)\neq(1,1)}(1+y_{i\alpha_1+j\alpha_2}^{p^{i(r+1)+j(s+1)-1}})^{c_{ij}}\\
&\times(1+y^{p^s}_{\alpha_2})(1+y^{p^r}_{\alpha_1}).
\end{align*}
And hence $[y^{p^r}_{\alpha_1}, y^{p^s}_{\alpha_2}]_{\circ} =c_{11}y^{p^{r+s+1}}_{\alpha_1+\alpha_2}$, $c_{11}\in\{\pm1,\pm2,\pm3\}$. 
This is the assertion (2). Note that assertion (2) also holds for $p=5$. The same proof works.

Now let $\alpha\in\Phi$, $\delta\in\Pi$ as before. Then
\begin{align*} [y^{p^r}_{\alpha}, y^{p^s}_{\delta}]=& x_{\alpha}(p^{r+1})[1- x_{\alpha}(-p^{r+1})h_{\delta}((1+p)^{p^s}) x_{\alpha}(p^{r+1})h_{\delta}((1+p)^{-p^s})]\\ &\times h_{\delta}((1+p)^{p^s}).
\end{align*}
By the relation R8 \cite[Page 30]{Steinberg},
we get
$$
x_{\alpha}((1+p)^{p^s\langle \alpha, \delta \rangle}p^{r+1})h_{\delta}((1+p)^{p^s})=h_{\delta}((1+p)^{p^s})x_{\alpha}(p^{r+1}).
$$
And hence
$$x_{\alpha}((1+p)^{p^s\langle \alpha, \delta \rangle}p^{r+1})=h_{\delta}((1+p)^{p^s})x_{\alpha}(p^{r+1})h_{\delta}((1+p)^{-p^s}).$$
So
\begin{align*}&x_{\alpha}(-p^{r+1})x_{\alpha}((1+p)^{p^s\langle \alpha, \delta \rangle}p^{r+1})=x_{\alpha}(-p^{r+1})h_{\delta}((1+p)^{p^s})x_{\alpha}(p^{r+1})h_{\delta}((1+p)^{-p^s}).\end{align*}
By invoking the relation R1 in \cite[Page 30]{Steinberg}, we see that
\begin{align*}&x_{\alpha}(p^{r+1}(1+p)^{p^s\langle \alpha, \delta \rangle}-p^{r+1})=x_{\alpha}(-p^{r+1})h_{\delta}((1+p)^{p^s})x_{\alpha}(p^{r+1})h_{\delta}((1+p)^{-p^s}).\end{align*}
This implies that
 $$
 [y^{p^r}_{\alpha}, y^{p^s}_{\delta}]=(1+y_{\alpha})^{p^r}[1-(1+y_{\alpha})^{p^r(1+p)^{p^s\langle \alpha, \delta \rangle}-p^r}](1+y_{\delta})^{p^r}.
 $$
We therefore have
$$ 
[y^{p^r}_{\alpha}, y^{p^s}_{\delta}]_{\circ}  = -\langle\alpha,\delta \rangle y^{p^{r+s+1}}_{\alpha}.
$$
This proves the assertion (3).

For the Chevalley group ${\rm SL}_2(\mathbb{Z}_p)$, we by \cite[Page 1025]{WeiBian1} know that
\begin{align*}
&\left [
\begin{array}{cc}
1 & -p^{r+1}\\
0 & 1
\end{array}\right ]\left [
\begin{array}{cc}
1 & 0\\
p^{s+1} & 1
\end{array}\right ] \left [
\begin{array}{cc}
1 & p^{r+1}\\
0 & 1
\end{array}\right ] \left [
\begin{array}{cc}
1 & 0\\
-p^{s+1} & 1
\end{array}\right ]
\notag\\=&\left [
\begin{array}{cc}
1  & -p^{2r+s+3}(1+p^{r+s+2})^{-1}\\
0 & 1
\end{array}\right ]
\left [
\begin{array}{cc}
(1+p^{r+s+2})^{-1} & 0\\
0 & 1+p^{r+s+2}
\end{array}\right ]\\&\times\left [
\begin{array}{cc}
1  &0\\
-p^{r+2s+3}(1+p^{r+s+2})^{-1} & 1
\end{array}\right ].
\end{align*}
Recall that we have a homomorphism\,(cf. \cite[Corollary 6 of Page 46]{Steinberg})
$$
{\rm SL}_2(\mathbb{Z}_p)\longrightarrow\langle \mathfrak{X}_{\alpha}, \mathfrak{X}_{-\alpha}\rangle
$$
$$
\left [
\begin{array}{cc}
1  & t\\
0 & 1
\end{array}\right ]\longmapsto x_{\alpha}(t),\ \ \left [
\begin{array}{cc}
1  & 0\\
t& 1
\end{array}\right ]\longmapsto x_{-\alpha}(t),\ \ \left [
\begin{array}{cc}
 t &0\\
0 & t^{-1}
\end{array}\right ]\longmapsto h_{\alpha}(t).
$$
So we conclude that
\begin{align*}
&x_{\alpha}(-p^{r+1})x_{-\alpha}(p^{s+1})x_{\alpha}(p^{r+1})x_{-\alpha}(-p^{s+1})\\=&
x_{\alpha}(-p^{2r+s+3}(1+p^{r+s+2})^{-1})h_{\alpha}((1+p)^{\beta})x_{-\alpha}(-p^{r+2s+3}(1+p^{r+s+2})^{-1}).
\end{align*}
One should note that  $$(1+p^{r+s+2})^{-1}=1-p^{r+s+2}+p^{2(r+s+2)}-p^{3(r+s+2)}+p^{4(r+s+2)}+\cdots.$$
It follows from the properties of $p$-adic integers that there exists one element $\beta$ such that $$(1+p^{r+s+2})^{-1}=(1+p)^{\beta},$$
where $\beta = \beta_0 + \beta_1p + \beta_2p^2 + \cdots+ \beta_{r+s}p^{r+s}+ \beta_{r+s+1}p^{r+s+1}+ \cdots $, $\beta_k\in \mathbb{Z} $ and
$0 \leq \beta_k\leq (p - 1)$. According to the expansion formula of $(  1 +p^{r+s+2})^{-1}$, we can compute all $\beta_k$. For instance,
$ \beta_0 = \beta_1=\cdots = \beta_{r+s}=0$, $\beta_{r+s+1}=p-1$, $\beta_{r+s+2}=(\frac{p^{r+s+1}-1}{2}-1)\mod p$. Now,
\begin{align*}
[y^{p^r}_{\alpha}, y^{p^s}_{-\alpha}]=&x_{\alpha}(p^{r+1})[1-x_{\alpha}(-p^{r+1})x_{-\alpha}(p^{s+1})x_{\alpha}(p^{r+1})x_{-\alpha}(-p^{s+1})]x_{\alpha}(p^{s+1})\\
=&x_{\alpha}(p^{r+1})[1-x_{\alpha}(-p^{2r+s+3}(1+p^{r+s+2})^{-1})h_{\alpha}((1+p)^{\beta})\\&\times x_{-\alpha}(-p^{r+2s+3}(1+p^{r+s+2})^{-1})]x_{\alpha}(p^{s+1}).
\end{align*}
Since $G$ is simply connected, by the proof of \cite[Corollary 5 in Page 44]{Steinberg}\, 
(see also \cite[Corollary of Lemma 28 in Page 44]{Steinberg}), 
we know that there exist integers $n_i$ such that
$$
h_{\alpha}((1+p)^{\beta})=\prod_{\delta_i\in\Pi}h_{\delta_i}((1+p)^{n_i\beta})\,\,(\alpha=\sum_{i=1}^{\ell} n_i\delta_i, \   \delta_i\in\Pi, \  \ell=|\Pi|)
$$
uniquely. Thus we get
\begin{align}
\label{eqnumber}
[y^{p^r}_{\alpha}, y^{p^s}_{-\alpha}]=&(1+y_{\alpha})^{p^r}[1-{(1+y_{\alpha}^{p^{2r+s+2}})}^{-(1+p^{r+s+2})^{-1}}\prod_{\delta_i\in\Pi}(1+y_{\delta_i}^{n_i\beta})
{(1+y_{-\alpha}^{p^{r+2s+2}})}^{-(1+p^{r+s+2})^{-1}}](1+y_{-\alpha})^{p^s}.\end{align}
In order to determine  the lowest degree term of $[y^{p^r}_{\alpha}, y^{p^s}_{-\alpha}]$, our next 
step is to find out how large the $n_i$'s can be. It should be remarked that the highest root\,
(the root in $\Phi$ with maximum height) denoted by $\alpha_{\max}$ has the maximum $n_i$'s\, 
(cf. \cite[Proposition 25 in Page 178]{Bourbaki}). That is, let us set $\alpha_{\max}=\underset{\delta_i\in \Pi}{\sum} n_i^{\max} \delta_i$
and let $\alpha$ be any root with $\alpha=\sum_{\delta_i\in\Pi}n_i \delta_i$. Then $n_i\leq n_i^{\max}$, $  i\in[1,\ell]$, $\ell=|\Pi|$. 
Now, one can check, case by case, for type $A_{\ell}$, $B_{\ell}$,
$C_{\ell}$, $D_{\ell}$ and exceptional Lie type algebras these $n_i^{\max}$ from Bourbaki table (cf. \cite[Chapter VI, Section 4]{Bourbaki}).
\begin{enumerate}
\item[(1)] Type $A_{\ell}$: $\underset{i}{\rm Max}\{n_i^{\max} \}=1$,
\item[(2)] Type $B_{\ell}\,(\ell\geq 2)$: $\underset{i}{\rm Max}\{n_i^{\max} \}=2$\,\,(Page 214 of \cite[Chapter IV, Section 5]{Bourbaki}),
\item[(3)] Type $C_{\ell}\,(\ell\geq 2)$: $\underset{i}{\rm Max}\{n_i^{\max} \}=2$\,\,(Page 216 of \cite{Bourbaki}),
\item[(4)] Type $D_{\ell}\,(\ell\geq 3)$: $\underset{i}{\rm Max}\{n_i^{\max} \}=2$\,(Page 220 of \cite{Bourbaki}),
\item[(5)] Type $E_6$: $\underset{i}{\rm Max}\{n_i^{\max} \}=3$\,\,(Page 229 of \cite{Bourbaki}),
\item[(6)] Type $E_7$: $\underset{i}{\rm Max}\{n_i^{\max} \}=4$\,\,(Page 227 of \cite{Bourbaki}),
\item[(7)] Type $E_8$: $\underset{i}{\rm Max}\{n_i^{\max} \}=6$\,\,(Page 226 of \cite{Bourbaki}),
\item[(8)] Type $F_4$: $\underset{i}{\rm Max}\{n_i^{\max} \}=4$\,\,(Page 223 of \cite{Bourbaki}),
\item[(9)] Type $G_2$: $\underset{i}{\rm Max}\{n_i^{\max} \}=3$\,\,(Page 232 of \cite{Bourbaki}).
\end{enumerate}
Let $p$ satisfies the following hypothesis.
\begin{equation}
\label{eq:hyp}
HYP_{\Phi}: p>\underset{i}{\rm Max}\{n_i^{\max} \}
\end{equation}
Note that $HYP_{\Phi}$ depends on the type of the Lie algebra. Also note that if $p>5$ (which was the hypothesis assumed in proposition \ref{xxsec3.1}) then $p$ satisfies $HYP_{\Phi}$ for all root types $\Phi$.

For primes $p$ satisfying $ HYP_{\Phi}$, the lowest degree terms of $[y^{p^r}_{\alpha}, y^{p^s}_{-\alpha}]$ 
will come from the part $\prod_{\delta_i\in\Pi}(1+y_{\delta_i})^{n_i\beta}$ because
$\beta_0=\beta_1=\cdots=\beta_{r+s}=0$, $\beta_{r+s+1}=p-1$ and $p^{2r+s+2}$ or $p^{r+2s+2}$ is 
strictly larger than $n_ip^{r+s+1}$, $  r,s\geq 0$, because $p$ satisfies $HYP_{\Phi}$ 


Now \begin{align*}\prod_{\delta_i\in\Pi}(1+y_{\delta_i})^{n_i\beta}&=\prod_{\delta_i\in\Pi}(1+y_{\delta_i})^{n_i(\beta_0+\beta_1p+\cdots+\beta_{r+s}p^{r+s}+\beta_{r+s+1}p^{r+s+1}+\cdots)}\\
&=\prod_{\delta_i\in\Pi}(1+y_{\delta_i})^{n_i(\beta_{r+s+1}p^{r+s+1}+\cdots)}\\
&=\prod_{\delta_i\in\Pi}(1+y_{\delta_i})^{n_i((p-1)p^{r+s+1}+\cdots)}.
\end{align*}
Therefore we obtain
$$[y^{p^r}_{\alpha}, y^{p^s}_{-\alpha}]_{\circ}=\sum_{i=1}^{\ell}n_iy_{\delta_i}^{p^{r+s+1}},\ \ \  \ell=|\Pi|.$$
This completes the proof.
\end{proof}

Let us extend Proposition \ref{xxsec3.1} to the cases of $p=3$ and $p=5$. 

We are going to deduce the lowest degree terms of the commutators whenever $p=5$ and $p=3$.

For the prime integer $p=5$, note that the relations (1), (2), (3) of Proposition \ref{xxsec3.1} remain true because 
$c_{11}\in \{\pm1,\pm2,\pm3\}$, 
$\langle\alpha,\delta \rangle\in \{0, \pm1, \pm2\}$. We only need to modify (4) of Proposition \ref{xxsec3.1} 
for $p=5$. Notice that (4) of Proposition \ref{xxsec3.1} is the lowest degree term of the commutator 
relation $[y_\alpha^{p^r}, y_{-\alpha}^{p^s}]_\circ$. By \ref{eqnumber} we know that
\begin{align}
[y^{p^r}_{\alpha}, y^{p^s}_{-\alpha}]=&(1+y_{\alpha})^{p^r}[1-{(1+y_{\alpha}^{p^{2r+s+2}})}^{-(1+p^{r+s+2})^{-1}}\prod_{\delta_i\in\Pi}(1+y_{\delta_i}^{n_i\beta})
{(1+y_{-\alpha}^{p^{r+2s+2}})}^{-(1+p^{r+s+2})^{-1}}](1+y_{-\alpha})^{p^s}.\end{align}

For the prime integer $p=5$, the argument given before to deduce the lowest degree term $[y_\alpha^{p^r}, y_{-\alpha}^{p^s}]_\circ$ 
of $[y_\alpha^{p^r}, y_{-\alpha}^{p^s}]$ only fails for type $E_8$ where $\underset{i}{\rm max}\, \{n_i^{\rm max}\}=6$ (That is $p=5$ does not satisfy $HYP_{\Phi}$). 
From \cite[Page 225]{Bourbaki}, we see that if $\alpha$ is a $+ve$ root and $\alpha=\sum_{i=1}^8 n_i\delta_i$ 
is the root system of type $E_8$, then there must exist some $n_i$'s such that $n_i$'s $<5$. In other words, all of 
$n_i$'s can not be $\geq 5$ for all $i\in [1, 8]$. Then, for type $E_8$ and prime integer $p=5$, the lowest degree term
$$
[y_\alpha^{5^r}, y_{-\alpha}^{5^s}]_\circ=\sum_{i=1}^\ell m_i y_{\delta_i}^{5^{r+s+1}},\ \  where \left\{
\begin{aligned}
 & m_i=n_i\  \text{if} \  n_i<5; \\
& m_i=0  \ \text{ if}\  n_i\geq 5;\\
& \alpha=\sum_i^8 n_i \delta_i
\end{aligned}
\right.
$$
We get the above 
lowest degree term by using the similar argument as (4) of Proposition \ref{xxsec3.1} using equation \ref{eqnumber} and noticing that $p^{2r+s+2}$ or $p^{r+2s+2}$ are less than or equal to $n_ip^{r+s+1}$ for $r=s=0$ and $n_i \geq 5$. The point is $[y_\alpha^{5^r}, y_{-\alpha}^{5^s}]_\circ$ contains only powers $y_{-}^{5^{r+s+1}}$ and this is what we will use in the future.
This completes the proof of case $p=5$. 

Now let us consider the case of $p=3$. It should be remarked that (1) and (3) of Proposition \ref{xxsec3.1} 
remain unchanged as $\langle\alpha,\delta \rangle\in \{0, \pm1, \pm2\}$. Let us look at $[y_{\alpha_1}^{p^r}, y_{\alpha_2}^{p^s}]_\circ$ 
if $\alpha_1+\alpha_2\in \Phi$, $\alpha_1\neq -\alpha_2$ (i.e. part (2) of Proposition \ref{xxsec3.1}) for $p=3$. 
Note that we have 
\begin{align*}
(1+y^{3^r}_{\alpha_1})(1+y^{3^s}_{\alpha_2})=&(1+y^{3^{r+s+1}}_{\alpha_1+\alpha_2})^{c_{11}}\prod_{i,j>0, (i,j)\neq(1,1)}(1+y_{i\alpha_1+j\alpha_2}^{3^{i(r+1)+j(s+1)-1}})^{c_{ij}}\\
&\times(1+y^{3^s}_{\alpha_2})(1+y^{3^r}_{\alpha_1}).
\end{align*}
Here $c_{ij}\in \{\pm1,\pm2,\pm3\}$. Suppose that $c_{11}\neq \pm3$. Then 
$$
[y^{3^r}_{\alpha_1}, y^{3^s}_{\alpha_2}]_\circ=c_{11} y_{\alpha_1+\alpha_2}^{3^{r+s+1}}.
$$
Suppose that $c_{11}=\pm3$. Since $(r+1)+(s+1)=r+s+2\leq i(r+1)+j(s+1)-1$ for all $(i, j)$ with $i, j>0$ and $(i, j)\neq (1, 1)$, 
we have 
\begin{align*}
[y^{3^r}_{\alpha_1}, y^{3^s}_{\alpha_2}]_\circ= & s_{11} y_{\alpha_1+\alpha_2}^{3^{r+s+2}}+\text{possibly some other terms of the form}\\
& s_{ij}y_{i\alpha_1+j\alpha_2}^{3^{r+s+2}}\  \text{for} \ i,j>0, (i, j)\neq (1,1)\ \text{depending on} \ c_{ij},
\end{align*}
where $s_{11}\neq 0$ and $s_{ij}$ for $i, j>0$, $(i, j)\neq (1,1)$ might be zero. 
Here also, the main point is not the exact values of the $s_{i,j}$'s but the fact that 
$[y^{3^r}_{\alpha_1}, y^{3^s}_{\alpha_2}]_\circ$ contains only powers $y_{-}^{3^{r+s+2}}$. This will later be used in the proof of our main theorem.

Let us next look at (4) of Proposition \ref{xxsec3.1} for $p=3$ and $p$ does not satify $HYP_{\Phi}$. We just need to find $[y_\alpha^{3^r}, y_{-\alpha}^{3^s}]_\circ$ 
for $p=3$. Again from Bourbaki classification of root systems and corresponding expressions for 
$\alpha_{\rm max}$ for type $F_4, E_6, E_7, E_8$ and $G_2$ \cite[Page 223--Page 232]{Bourbaki}, we see 
that if $\alpha\in \Phi$, $\alpha$ $+ve$ root, $\alpha=\sum_{i=1}^\ell n_i \delta_i$, then there must 
exist some $n_i$ such that $n_i\leq 2$. So, by the same argument as for the case $p=5$ and type $E_8$, we 
get
$$
[y_\alpha^{3^r}, y_{-\alpha}^{3^s}]_\circ=\sum_{i=1}^\ell m_i y_{\delta_i}^{3^{r+s+1}},\ \  where \left\{
\begin{aligned}
 & m_i=n_i\  \text{if} \  n_i\leq 2; \\
& m_i=0  \ \text{ if}\  n_i\geq 3;\\
& \alpha=\sum_i^l n_i \delta_i\, .
\end{aligned}
\right.
$$
This completes the computation of the lowest degree terms for the prime integer $p=3$.


\section{Partial Differential Equations}
\label{xxsec4}

Let
\begin{equation}\label{eq:multiindex}
\omega=\sum_{ n_{\underline{\xi}}}^{\rm finite}a_{\underline{\xi}}\prod_{\xi\in\Phi^-, \Pi, \Phi^+}(y_{\xi}^{p^s})^{ n_{\xi}},
\end{equation}
where the sum is finite, $ n_{\underline{{\xi}}}=(( n_{\xi}), \xi\in \Phi^-, \Pi, \Phi^+)$, $ n_{\xi}\in \mathbb{N}\cup \{0\}$, $a_{\underline{\xi}}\in \mathbb{F}_p$,
$\underline{\xi}=( \xi\in \Phi^-, \Pi, \Phi^+)$, i.e., $\underline{\xi}$ is just the collection of the 
$-ve$, simple and $+ve$ roots ordered according to the order of Lazard basis,
i.e. according to an order compatible with increasing height function on the roots.

\textbf{Explanation of multi-index notation in \eqref{eq:multiindex}}:
	Before proceeding, for the convenience of the reader,  let us explain our multi-index notation \eqref{eq:multiindex} in a little more details because we will be using it throughout the text. For example, let us take the particular case when $G=SL_2(\mathbb{Z}_p)$. 
Then the only negative root is $\xi = (2,1) \in \Phi^-$ and the corresponding element $y_\xi = \left[ \begin{matrix} 0 & 0  \\ p & 0  \end{matrix} \right]$. Therefore working over $\mathbb{F}_p$,  $y_\xi ^{p^s}= \left[ \begin{matrix} 0 & 0  \\ p^{s+1} & 0  \end{matrix} \right]$.

  For the simple root $\xi  = (1,2) \in \Pi$, the corresponding element $y_\xi = \left[ \begin{matrix} (1+p)-1 & 0  \\  0 & (1+p)^{-1} -1  \end{matrix} \right]$. Hence $y_\xi ^{p^s}= \left[ \begin{matrix} (1+p)^s-1 & 0  \\  0 & (1+p)^{-s} -1  \end{matrix} \right]$.

  For the positive root $\xi = (1,2) \in \Phi^+$, the corresponding element $y_\xi = \left[ \begin{matrix} 0 & p  \\ 0 & 0  \end{matrix} \right]$.  Therefore, $y_\xi ^{p^s}= \left[ \begin{matrix} 0 & p^{s+1}  \\ 0 & 0  \end{matrix} \right]$.
  
  So if $G$ is $SL_2$, then $\omega$ denotes just a finite sum of monomials of the form $\prod_{\xi\in\Phi^-, \Pi, \Phi^+}(y_{\xi}^{p^s})^{ n_{\xi}}$ (where $n_\xi$ are non-negative integers)  multiplied  with  scalars $a_{\underline{\xi}}$ depending on each monomial. This is certainly an element in the mod-$p$ Iwasawa algebra of $G(1)$.

Let $\gamma \in\Phi$ or $\Pi$. Our goal in this section is to show the following lemma which
will be crucial to show our main theorem\,(ref. Theorem \ref{xxsec5.1}).
\begin{lemma}\label{xxsec4.1}
$$
[y^{p^r}_{\gamma}, \omega]_{\circ}=\sum_{
 \eta \ \text{such that}\
 [y^{p^r}_{\gamma}, y^{p^s}_{\eta}]_{\circ}\neq 0}\frac{\partial\omega}{\partial y^{p^s}_{\eta}}[y^{p^r}_{\gamma}, y^{p^s}_{\eta}]_{\circ},$$
where the sum is over all roots $\eta$ such that $[y^{p^r}_{\gamma}, y^{p^s}_{\eta}]_{\circ}\neq 0$.
\end{lemma}

\begin{proof}
$$
[y^{p^r}_{\gamma}, \omega]_{\circ}= \left \{ \sum_{ n_{\underline{\xi}}}a_{\underline{\xi}}y^{p^r}_{\gamma}(\prod_{\xi\in \Phi^-, \Pi, \Phi^+}{(y^{p^s}_{\xi})}^{ n_{\xi}})- \sum_{ n_{\underline{\xi}}}a_{\underline{\xi}}    (\prod_{\xi}{(y^{p^s}_{\xi})}^{ n_{\xi}}) y^{p^r}_{\gamma}          \right\}_{\circ}.$$
Suppose $\eta$ is the first root in the ordering such that  $[y^{p^r}_{\gamma}, y^{p^s}_{\eta}]_{\circ}\neq 0$. In the following, we are going to change $y^{p^r}_{\gamma}y^{p^s}_{\eta} \rightarrow y^{p^s}_{\eta}y^{p^r}_{\gamma}$ popping out a $[y^{p^r}_{\gamma}, y^{p^s}_{\eta}]$ term.
We  obtain
\begin{align}[y^{p^r}_{\gamma}, \omega]_{\circ}&= \Bigg\{ \sum_{ n_{\underline{\xi}}}a_{\underline{\xi}}\bigg(\prod_{\xi<\eta}{(y^{p^s}_{\xi})}^{ n_{\xi}} y^{p^s}_{\eta} y^{p^r}_{\gamma}{(y^{p^s}_{\eta})}^{n_{\eta^{-1}}}\prod_{\xi>\eta}{(y^{p^s}_{\xi})}^{n_{\xi}}\bigg)\\&
-\sum_{ n_{\underline{\xi}}}a_{\underline{\xi}}(\prod_{\xi\in \Phi^-, \Pi, \Phi^+}{(y^{p^s}_{\xi})}^{ n_{\xi}})y^{p^r}_{\gamma}\Bigg\}_{\circ} \\
&+\sum_{ n_{\underline{\xi}}}a_{\underline{\xi}}(\prod_{\xi<\eta}{(y^{p^s}_{\xi})}^{ n_{\xi}}[y^{p^r}_{\gamma}, y^{p^s}_{\eta}]_{\circ}(y^{p^s}_{\eta} )^{n_{\eta-1}}\prod_{\xi>\eta}( y^{p^s}_{\xi} )^{n_{\xi}}.
\end{align}
(Here $\xi <\eta$ means all the roots $\xi$ which come before the root $\eta$ in the total order on the roots compatible with the ordering in  Lazard's order basis).
Iterating the above step until $y^{p^r}_{\gamma} $ passes through all $y^{p^s}_{\eta} $ in (4.1) 
and pops out $[y^{p^r}_{\gamma}, y^{p^s}_{\eta}]_{\circ}$ term in (4.3). Then we get
\begin{align}[y^{p^r}_{\gamma}, \omega]_{\circ}&=\Bigg\{ \sum_{ n_{\underline{\xi}}}a_{\underline{\xi}}(\prod_{\xi<\eta}{(y^{p^s}_{\xi})}^{ n_{\xi}} (y^{p^s}_{\eta} )^{n_{\eta}} y^{p^r}_{\gamma}\prod_{\xi>\eta}{(y^{p^s}_{\xi})}^{n_{\xi}}\\&
-\sum_{ n_{\underline{\xi}}}a_{\underline{\xi}}(\prod_{\xi\in \Phi^-, \Pi, \Phi^+}{(y^{p^s}_{\xi})}^{ n_{\xi}})y^{p^r}_{\gamma}\Bigg\}_{\circ} \\
&+\sum_{ n_{\underline{\xi}}}a_{\underline{\xi}}(\prod_{\xi<\eta}{(y^{p^s}_{\xi})}^{ n_{\xi}}[y^{p^r}_{\gamma}, y^{p^s}_{\eta}]_{\circ}(y^{p^s}_{\eta} )^{n_{\eta-1}}\prod_{\xi>\eta}( y^{p^s}_{\xi} )^{n_{\xi}}\notag\\
&+\sum_{ n_{\underline{\xi}}}a_{\underline{\xi}}(\prod_{\xi<\eta}{(y^{p^s}_{\xi})}^{ n_{\xi}}y^{p^s}_{\eta}[y^{p^r}_{\gamma}, y^{p^s}_{\eta}]_{\circ}(y^{p^s}_{\eta} )^{n_{\eta-2}}\prod_{\xi>\eta}( y^{p^s}_{\xi} )^{n_{\xi}}\notag\\
&+\cdots \cdots \cdots \notag\\
& +\sum_{ n_{\underline{\xi}}}a_{\underline{\xi}}(\prod_{\xi<\eta}{(y^{p^s}_{\xi})}^{ n_{\xi}}(y^{p^s}_{\eta} )^{n_{\eta-1}}[y^{p^r}_{\gamma}, y^{p^s}_{\eta}]_{\circ}\prod_{\xi>\eta}( y^{p^s}_{\xi} )^{n_{\xi}}
\end{align}
Applying Proposition \ref{xxsec3.1}, we have
$$[[y^{p^r}_{\gamma}, y^{p^s}_{\eta}]_{\circ},y^{p^s}_{\xi}]_{\circ} \  \text{is either } 0 \text{ or has terms of degree } p^{r+2s+2}.$$ But degree of $[y^{p^r}_{\gamma}, y^{p^s}_{\eta}]_{\circ}y^{p^s}_{\xi}$ is $p^{r+2s+1}$. 
Therefore,
$$
[y^{p^r}_{\gamma}, y^{p^s}_{\eta}]_{\circ}y^{p^s}_{\xi}=y^{p^s}_{\xi}[y^{p^r}_{\gamma}, y^{p^s}_{\eta}]_{\circ} \text{ modulo higher-degree terms }.
$$


We therefore get
\begin{align}[y^{p^r}_{\gamma}, \omega]_{\circ}&=\bigg\{\sum_{ n_{\underline{\xi}}}a_{\underline{\xi}}(\prod_{\xi<\eta}{(y^{p^s}_{\xi})}^{ n_{\xi}} {(y^{p^s}_{\eta})}^{ n_{\eta}} y^{p^r}_{\gamma}\prod_{\xi>\eta}{(y^{p^s}_{\xi})}^{n_{\xi}})\notag\\&
-\sum_{ n_{\underline{\xi}}}a_{\underline{\xi}}(\prod_{\xi\in \Phi^-, \Pi, \Phi^+}{(y^{p^s}_{\xi})}^{ n_{\xi}})y^{p^r}_{\gamma}\bigg\}_{\circ}+\frac{\partial\omega}{\partial y^{p^s}_{\eta}}[y^{p^r}_{\gamma},y^{p^s}_{\eta}]_{\circ}.
\end{align}
Note that here we are working under the background of the lowest-degree
terms. So the ``=" sign does make sense against the backdrop of $[-,-]_{\circ}$.

 Iterating this method until $y^{p^s}_{\gamma}$ passes  through all $y^{p^s}_{\xi}$ for all $\xi>\eta$, we finally arrive at
$$[y^{p^r}_{\gamma}, \omega]_{\circ}=\sum_{\eta \ \text{such that}\  [y_{\gamma}, y_{\eta}]_{\circ}\neq0}\frac{\partial\omega}{\partial y^{p^s}_{\eta}}[y^{p^r}_{\gamma},y^{p^s}_{\eta}]_{\circ} .$$
This completes proof of this lemma.

\end{proof}

\section{Main Result and Its Proof}\label{xxsec5}

In this section we will state and prove our main result assuming Claim \ref{xxsec5.3}, which we will prove 
in the next section. Let $G$ be a semi-simple, simply connected, split Chevalley group 
over $\mathbb{Z}_p$, $G(1)$ be the first congruence kernel of $G$ and $\Omega_{G(1)}$ 
be the mod-$p$ Iwasawa algebra of $G(1)$ over $\mathbb{F}_p$. Let us recall that an element $r\in\Omega_{G(1)}$ is said to be
\textit{normal} if  $r\Omega_{G(1)}=\Omega_{G(1)}r$. Clearly, $\mathbb{F}_p$ and 
those elements of $\Omega_{G(1)}$ which contain constant terms are normal elements. 
It is a natural question whether the converse statement of this result holds true. 
Our main objective is to determine which elements in the $\Omega_{G(1)}$ are eligible 
for normal elements. Our main theorem is the following

\begin{theorem}\label{xxsec5.1}
Let $p$ be a prime with $p\geq 3$ and $G$ be a semi-simple, simply connected, split Chevalley group 
over $\mathbb{Z}_p$. Suppose that $G$ is one of the following Chevalley groups of Lie type: 
$A_{\ell}\, (\ell\geq 1), B_{\ell}\,(\ell\geq 2), C_{\ell}\,(\ell\geq 2), D_{\ell}\,(\ell\geq 3), E_6, E_7, E_8, F_4, G_2$.
Then for any nonzero element $W\in \Omega_{G(1)}$, $W$ is a normal element if and only if 
$W$, as a noncommuttaive formal power series, contains constant terms. In this case, $W$ is a unit. 
\end{theorem}

\begin{proof}
Let $W$ be a nonzero element of $W\in \Omega_{G(1)}$. 
Suppose that $W$, as a noncommutative formal power series, contains constant terms. 
Then it is straightforward to check that $W$ is invertible. And hence $W$ is a normal 
element of $\Omega_{G(1)}$. In this case, $W$ is a unit.

Let $W$ be a nonzero normal element of $\Omega_{G(1)}$ and $W$ is of the form
$$
W = w_m + w_{m+1} + w_{m+2}+ \cdots+ w_d + \cdots, 
$$
where $w_d$\,($d=m,m+1,m+2,\cdots, m\geq1$) are homogeneous polynomials with 
respect to $y_{\xi}$\,($\xi$ varying through $\Phi$ and $\Pi$) of degree $d$.
\begin{equation} 
w_d =\sum_{m^\prime_{\underline{\xi}}}a_{\underline{\xi}}\prod_{\xi\in \Phi^-, \Pi, \Phi^+}y_{\xi}^{m^\prime_{\xi}},
\end{equation}
(the multi-index notation has same meaning as in \eqref{eq:multiindex} with explanations given after \eqref{eq:multiindex}).

 Moreover, we put
\begin{align*}
s_d=\max\{\ s\ | \ &p^s \text{ is a common divisor of the elements of }\  m'_{\xi}, \\& \ \text{ appearing in (5.1) with}\  m'_{\xi}\neq 0\ \},
\end{align*}
which will be frequently invoked in the sequel.

Since $W$ is a normal element, there exists an element $D_{\gamma}(r)\in \Omega_{G(1)}$ such that
\begin{equation}[y^{p^r}_{\gamma}, W]= W\cdot D_{\gamma}(r)\end{equation}
for each $\gamma\in \Phi$ or $\gamma\in\Pi$. We define
$$
s=\min\{\ s_d\ |\ d=m,m+1,m+2,\cdots \ \}.
$$

We divide the proof of this theorem into two cases:
$s=s_m$ and $s<s_m$.

\vspace{2mm}

\textbf{Case 1:} $s=s_m$.

\vspace{2mm}

In  this case, we by (5.2) get
\begin{equation}[y^{p^r}_{\gamma}, w_m]_{\circ}= w_m\cdot (D_{\gamma}(r))_{\circ}\end{equation}
for each $y_{\gamma}$\,($\gamma\in \Phi$ or $\gamma\in\Pi$). Recall that 
$[y^{p^r}_{\gamma}, w_m]_{\circ}$ and $(D_{\gamma}(r))_{\circ}$ stand for the lowest degree terms in
$[y^{p^r}_{\gamma}, w_m]$ and $D_{\gamma}(r)$, respectively. It should be 
pointed out that $[y^{p^r}_{\gamma}, w_m]_{\circ}$ is a homogenous polynomial of 
degree $m-p^s+p^{r+s+1}$.

We can assume that the lowest-degree homogenous polynomial $w_m$ of $W$ is of the form
\begin{equation}
w_m=\sum_{ n_{\underline{\xi}}}^{\rm finite}a_{\underline{\xi}}\prod_{\xi\in\Phi^-, \Pi, \Phi^+}(y_{\xi}^{p^s})^{ n_{\xi}},
\end{equation}
(as in the beginning of Section \ref{xxsec4}) such that 
$$
w_m\in \mathbb{F}_p[y_{\xi}^{p^s},   \xi\in\Phi^-, \Pi, \Phi^+]\setminus\mathbb{F}_p[y_{\xi}^{p^{s+1}}, \ \xi\in\Phi^-, \Pi, \Phi^+],
$$ 
where $\mathbb{F}_p[y_{\xi}^{p^s}, \ \xi\in\Phi^-, \Pi, \Phi^+]$ denotes the polynomial 
ring generated by  $y_{\xi}^{p^s}$ for $\xi\in\Phi^-, \Pi, \Phi^+$ over the field $\mathbb{F}_p$. 
It follows from Lemma \ref{xxsec4.1} that 
$$
[y^{p^r}_{\gamma}, w_m]_{\circ}=\sum_{\eta \ \text{such that}\  [y^{p^r}_{\gamma}, y^{p^s}_{\eta}]_{\circ}\neq0}\frac{\partial w_m}{\partial y^{p^s}_{\eta}}[y^{p^r}_{\gamma}, y^{p^s}_{\eta}]_{\circ}.
$$

\begin{claim}\label{xxsec5.2}
$\frac{\partial w_m}{\partial y^{p^s}_{\gamma}}$ are not exactly zeroes for all $\gamma\in \Phi $ or $\Pi $.
\end{claim}

\begin{proof}
In view of (5.3), we can write
$$w_m=\sum_{ n_{\gamma}=0}^{a_{\gamma}}(y_{\gamma}^{p^s})^{ n_{\gamma}}v_{n_{\gamma}}(\underline{y}^{\gamma}),$$
where $v_{n_{\gamma}}(\underline{y}^{\gamma})$ is a polynomial over $\mathbb{F}_p$ in $y_{\xi}^{p^s}$ for $\xi \in  \Phi\cup\Pi\setminus \{\gamma\} $.
Suppose on the contrary that
$$
\frac{\partial w_m}{\partial y^{p^s}_{\gamma}}=n_{\gamma}\sum_{ n_{\gamma}=1}^{a_{\gamma}}(y_{\gamma}^{p^s})^{ n_{\gamma}-1}v_{n_{\gamma}}(\underline{y}^{\gamma})=0,\ \  \, \gamma\in \Phi  \ \text{or}\   \Pi.
$$
Considering the polynomial above related to $y^{p^s}_{\gamma}$, we see that 
$n_{\gamma}v_{n_{\gamma}}(\underline{y}^{\gamma})=0$, and we therefore get 
$p\,|\,n_{\gamma}$, $ \gamma\in \Phi $ or $\Pi $. But then $w_m\in\mathbb{F}_p[y_{\xi}^{p^{s+1}}, \ \xi\in\Phi^-, \Pi, \Phi^+]$, 
which is a contradiction. This proves this claim.

\end{proof}
Now we complete the proof of Theorem \ref{xxsec5.1} for the case $s=s_m$. By Claim \ref{xxsec5.2}, 
there exists $ \gamma_1\in\Phi$ or $\gamma_1\in\Pi$ such that $\frac{\partial w_m}{\partial y^{p^s}_{\gamma_1}}\neq 0$. 
Certainly there exists $\alpha \in\Phi$ or $\Pi$ such that $[y^{p^r}_{\alpha}, y^{p^s}_{\gamma_1}]_{\circ}\neq0$\,(if $\gamma_1\in \Phi$, one can 
choose $\alpha\in\Phi$ such that $\alpha+\gamma_1\in \Phi$; and if $\gamma_1\in\Pi$, we can 
choose any $\alpha\in\Phi$ and use Proportion \ref{xxsec3.1}). 
Then
$$
w_m\cdot (D_{\alpha}(r))_{\circ}=[y^{p^r}_{\alpha},  w_m]_{\circ}=\frac{\partial w_m}{\partial y^{p^s}_{\gamma}}[y^{p^r}_{\alpha}, y^{p^s}_{\gamma_1}]_{\circ}+\text{other terms}\,(\text{cf. Lemma \ref{xxsec4.1}}).
$$
By Proposition \ref{xxsec3.1}, for $p>5$, we know that
$$ [y^{p^r}_{\alpha}, y^{p^s}_{\gamma_1}]_{\circ}=\left\{
\begin{aligned}
 &c_{11}y_{\alpha+\gamma_1}^{p^{r+s+1}}, \ \text{ if} \ \alpha\in\Phi, \gamma\in \Phi ,  \alpha+\gamma_1\in\Phi ,  \alpha\neq -\gamma_1 ; \\
& -\langle \alpha,\gamma_1\rangle y_{\alpha}^{p^{r+s+1}},  \ \text{ if}\  \alpha\in\Phi, \gamma_1\in \Pi .
\end{aligned}
\right.
$$
Suppose we have the first one, i.e. $[y^{p^r}_{\alpha}, y^{p^s}_{\gamma_1}]_{\circ}=c_{11}y_{\alpha+\gamma_1}^{p^{r+s+1}}$. Then
\begin{align}
w_m\cdot (D_{\alpha}(r))_{\circ}&=\frac{\partial w_m}{\partial y^{p^s}_{\gamma_1}}[y^{p^r}_{\alpha}, y^{p^s}_{\gamma_1}]_{\circ}+\text{other terms}\\
&=\frac{\partial w_m}{\partial y^{p^s}_{\gamma_1}}c_{11}y_{\alpha+\gamma_1}^{p^{r+s+1}}+\text{other terms}\neq0\,(\text{by the choice of $\gamma_1$})\notag.
\end{align}
Comparing the coefficient of $y_{\alpha+\gamma_1}^{p^{r+s+1}}$ of the above relation on 
both sides we see that
\begin{equation}\frac{\partial w_m}{\partial y^{p^s}_{\gamma_1}}c_{11}y_{\alpha+\gamma_1}^{p^{r+s+1}}=w_m\cdot U(\underline{y})\neq0,
\end{equation}
where $U(\underline{y})$ is the polynomial form coming from $(D_{\alpha}(r))_{\circ}$. Now comparing the degree of $y_{\gamma_1}^{p^s}$ on both sides of relation (5.6) we immediately arrive
at a contradiction\,(Note that in the above argument we can also choose $\alpha$ to be $-\gamma_1$ and similar contradiction will arise). Now  if
$$[y^{p^r}_{\alpha}, y^{p^s}_{\gamma_1}]_{\circ}=-\langle \alpha, \gamma_1\rangle y_{\alpha}^{p^{r+s+1}},$$ by the same argument as before\,(replacing $c_{11}y_{\alpha+\gamma_1}^{p^{r+s+1}}$ in the argument above by $-\langle \alpha, \gamma_1\rangle y_{\alpha}^{p^{r+s+1}}$) we arrive at a contradiction by comparing degrees of
$y^{p^s}_{\gamma_1}$. The argument for $p=3,5$ extending Proposition \ref{xxsec3.1} is the same. This implies that $W$, as a non-zero normal element in $\Omega_{G(1)}$, 
must contain constant terms and  completes the proof of Theorem 5.1 for the case $s=s_m$.
\vspace{2mm}

\textbf{Case 2.} $s<s_m$.

\vspace{2mm}

Now there exists some fixed $d$ with $d>m$ such that $s=s_d<s_m$, and it follows from (5.2) that
$$[y^{p^r}_{\gamma},w_d]_{\circ}=w_m\cdot ( D_{\gamma}(r))_{\circ}$$
for  each $y_{\gamma}$\,($\gamma\in \Phi$ or $\gamma\in\Pi$) provided $r\gg0$.

To proceed our discussion, we assume that $w_d$ is of the form
\begin{equation}w_d=\sum_{n_{\underline{\xi}}=0}^{p-1}\prod_{\ \ \xi\in\Phi^-, \Pi, \Phi^+}(y_{\xi}^{p^s})^{n_{\xi}}h_{n_{\underline{\xi}}}(\underline{y}^{p^{s+1}}), 
\end{equation}
where the sum is over all the component $n_{\xi}$ of $n_{\underline{\xi}}$ varying from $0$ to $p-1$. Here 
$$w_d \in \mathbb{F}_p[y_{\xi}^{p^s},   \xi\in\Phi^-, \Pi, \Phi^+]\setminus\mathbb{F}_p[y_{\xi}^{p^{s+1}}, \ \xi\in\Phi^-, \Pi, \Phi^+],
$$ 
where $\mathbb{F}_p[y_{\xi}^{p^s}, \ \xi\in\Phi^-, \Pi, \Phi^+]$ denotes the polynomial 
ring generated by  $y_{\xi}^{p^s}$ for $\xi\in\Phi^-, \Pi, \Phi^+$ over the field $\mathbb{F}_p$. 

\begin{claim}\label{xxsec5.3}
$w_m\,|\,\frac{\partial w_d}{\partial y^{p^s}_{\gamma}}$ for all $ \gamma\in\Phi$ or $\gamma\in\Pi$.
\end{claim}
Assume that Claim \ref{xxsec5.3} holds true now. We are going to prove it later in Section \ref{xxsec6} after finishing the proof of our main theorem.
Let us simplify our notations and denote the polynomial ring $\mathbb{F}_p[y_{\xi}^{p^s}, \ \xi\in\Phi^-, \Pi, \Phi^+]$ by $\mathbb{F}_p[\underline{y}^{p^s}] $.
\begin{claim}\label{xxsec5.4}
For $w_m, w_d$, there exist $ u\in \mathbb{F}_p[\underline{y}^{p^s}] $ and $v\in  \mathbb{F}_p[\underline{y}^{p^{s+1}}]$ such that $w_d=w_mu+v$
\end{claim}
\begin{proof}
According to Claim \ref{xxsec5.3}, there exists $ \,u_{\gamma}\in \mathbb{F}_p[\underline{y}^{p^s}] $ such that
 \begin{equation}\frac{\partial w_d}{\partial y^{p^s}_{\gamma}}=w_mu_{\gamma},\ \    \gamma\in \Phi,\,\Pi.\end{equation}
 Suppose that $u_{\gamma}$ is of the following form:
 \begin{equation}u_{\gamma}=\sum_{n_{\underline{\xi}}=0}^{p-1}\prod_{\ \ \xi\in\Phi^-, \Pi, \Phi^+}(y_{\xi}^{p^s})^{n_{\xi}}g_{n_{\underline{\xi}}}^{\gamma}(\underline{y}^{p^{s+1}}).
 \end{equation}
Using (5.7) we get $$\frac{\partial w_d}{\partial y^{p^s}_{\gamma}}=\sum_{\substack{n_{\xi}=0\\ (\xi<\gamma)}}^{p-1}\sum_{n_{\gamma}=1}^{p-1}\sum_{\substack{n_{\xi}=0\\ (\xi>\gamma)}}^{p-1}n_{\gamma}
 \prod_{\xi<\gamma}(y_{\xi}^{p^s})^{n_{\xi}}(y_{\gamma}^{p^s})^{n_{\gamma-1}} \prod_{\xi>\gamma}(y_{\xi}^{p^s})^{n_{\xi}}h_{n_{\underline{\xi}}}(\underline{y}^{p^{s+1}}).$$
 Here we need to keep tract of the components $n_\xi $ of $n_{\underline{\xi}}$ for the roots $\xi $ coming before $\gamma$, that is $\xi <\gamma$ (in the total order on the roots according to Lazard's ordered basis), $\xi =\gamma$ and $\xi>\gamma$. 
 We also include this in the subindices of the polynomial    $h_{n_{\underline{\xi}}}(\underline{y}^{p^{s+1}})$. In the following we write  $h_{n_{\underline{\xi}}}(\underline{y}^{p^{s+1}})$  as $h_{((n_{\xi\,(\xi<\gamma)},n_{\gamma},n_{\xi\,(\xi>\gamma)}))}(\underline{y}^{p^{s+1}})$.
 So\begin{align}
 \frac{\partial w_d}{\partial y^{p^s}_{\gamma}}=&\sum_{\substack{n_{\xi}=0\\ (\xi<\gamma)}}^{p-1}\sum_{n_{\gamma}=0}^{p-2}\sum_{\substack{n_{\xi}=0\\ (\xi>\gamma)}}^{p-1}(n_{\gamma}+1)
 \prod_{\xi<\gamma}(y_{\xi}^{p^s})^{n_{\xi}}(y_{\gamma}^{p^s})^{n_{\gamma}}\notag\\&\times \prod_{\xi>\gamma}(y_{\xi}^{p^s})^{n_{\xi}}h_{((n_{\xi\,(\xi<\gamma)},n_{\gamma}+1,n_{\xi\,(\xi>\gamma)}))}(\underline{y}^{p^{s+1}}).
 \end{align}
 Taking (5.9) and (5.10) into (5.8) yields
 $$
 w_mg_{\underline{\xi}}^{\gamma}(\underline{y}^{p^{s+1}})=(n_{\gamma}+1)h_{((n_{\xi\,(\xi<\gamma)},n_{\gamma}+1,n_{\xi\,(\xi>\gamma)}))}(\underline{y}^{p^{s+1}}).
 $$
 This shows that
 $$
 w_m\,|\,h_{n_{\underline{\xi}}}(\underline{y}^{p^{s+1}}),
 $$
 where the components of $n_{\underline{\xi}}=(n_{\xi\,(\xi<\gamma)},n_{\gamma},n_{\xi\,(\xi>\gamma)})$ are not all zeros, 
 i.e. $n_{\underline{\xi}}\neq \underline{0}$ ($\underline{0}$ is the zero vector). That is, for each $h_{n_{\underline{\xi}}}(\underline{y}^{p^{s+1}})$, there exists a corresponding $h^*_{n_{\underline{\xi}}}(\underline{y}^{p^{s+1}})$ such that
 $$
 h_{n_{\underline{\xi}}}(\underline{y}^{p^{s+1}})=w_mh^*_{n_{\underline{\xi}}}(\underline{y}^{p^{s+1}})\, \, (n_{\underline{\xi}}\neq \underline{0}).
 $$
Taking this back into (5.7), we arrive at
 $$w_d=w_m\sum_{n_{\underline{\xi}\neq \underline{0}}}\, \prod_{\xi\in\Phi^-, \Pi, \Phi^+}(y_{\xi}^{p^s})^{n_{\xi}}{h^*}_{n_{\underline{\xi}}}(\underline{y}^{p^{s+1}})+h_{(0,0,\cdots,0)}(\underline{y}^{p^{s+1}}),$$
 where the summation is taken over $n_{\underline{\xi}}$ such that the components $n_{\xi}$ varies from $0$ to $p-1$, but all the components $n_{\xi}$ can not be simultaneously
 zero. Let us put 
 $$
 u=\sum_{ n_{\underline{\xi}}\neq\underline{0}}\prod_{\ \ \xi\in\Phi^-, \Pi, \Phi^+}(y_{\xi}^{p^s})^{n_{\xi}}{h^*}_{n_{\underline{\xi}}}(\underline{y}^{p^{s+1}}),\ \ v=h_{(0,0,\cdots,0)}(\underline{y}^{p^{s+1}}).
 $$ 
 This proves that 
 $$
 w_d=w_mu+v,
 $$
 where $u\in \mathbb{F}_p[\underline{y}^{p^s}] $ and $v\in  \mathbb{F}_p[\underline{y}^{p^{s+1}}]$ 
 and Claim \ref{xxsec5.4} follows from Claim \ref{xxsec5.3}.
 \end{proof}

Now we continue the proof for the Case $w<s_m$ of Theorem \ref{xxsec5.1}. 
Let us consider the following subset of $\Omega_{G(1)}$:
$$\begin{aligned}N(w_m) =& \{\ W\ |\ W \ \text{is a nontrivial normal element  }\\&   \text{with the lowest degree term}\   w_m, s(W)= s_m-1\ \},
\end{aligned}
$$
 where $s(W)$ is the $s$ corresponding to
$W$.
For any $W \in N(w_m)$, we assume that $s(W)=s_d$  for some $d > m$. Thus one can
write $W$ as
$$W = w_m + w_{m+1} + w_{m+2}+ \cdots+ w_d + \cdots. $$
Then $w_d=w_mu+v$,
where $u\in \mathbb{F}_p[\underline{y}^{p^{s_m-1}}] $ and $v\in  \mathbb{F}_p[\underline{y}^{p^{s_m}}]$. For convenience, we denote the index of $w_d$ by $d(W)$.
Let us write $W = W_0$ and $W_1 = W(1 - u)$. Then
$$\begin{aligned}
&W_1= w_m + w_{m+1} + w_{m+2} + \cdots + (w_d  - w_mu) + (w_{d+1}  - w_{m+1}u)\\
&\ \ \ \ \ \ +(w_{d+2}- w_{m+2}u) + \cdots  + (w_{2d-m}- w_du) + \cdots.
\end{aligned}
$$
It is easy to verfiy that $W_1 \in N(w_m)$ and $d(W_0) < d(W_1)$.
Likewise, for $ W_1$,
 there exist $ u'$ and $v'$ such that $w_{1d} = w_mu'+v'$, where
$w_{1d}$ is the first homogeneous
polynomial satisfying the condition  $s(W_1)= s_m-1$ in $W_1$. Let us set $W_2 = W_1(1- u')$. It
is also easy to check that
$W_2 \in N(w_m)$ and $d(W_1) < d(W_2)$. Repeating this
process continuously, we finally construct an infinite sequence of normal elements
$$W_0 = W, \ W_1 = W(1 - u),  \ \text{and} \ W_2 = W(1 - u)(1 - u'),\  \cdots.$$
Let us set $\underset{n\rightarrow \infty}{\rm lim}\, W_n = V$. Then $V$  is a normal element with the form
$$V =  v_m + v_{m+1} +\cdots  v_{d-1} + v_{d}+ \cdots ,$$
where $v_m=w_m$. It follows that $s(V)> s_{m-1}$, a contradiction. This shows that $W$, 
as a non-zero normal element in $\Omega_{G(1)}$, 
must contain constant terms  in the  case that $s<s_m$. This completes the proof of our main 
theorem\,(Theorem \ref{xxsec5.1}) provided the Claim \ref{xxsec5.3} is true.

\end{proof}
\section{Proof of Claim \ref{xxsec5.3}}
\label{xxsec6}

We will give a detailed proof for Claim \ref{xxsec5.3} in this section. 
Let us restate Claim 5.3:  $w_m\,|\,\frac{\partial w_d}{\partial y^{p^s}_{\gamma}}$, for all $ \,\gamma\in\Phi$ or $\gamma\in\Pi$.

\begin{proof}
Let $\gamma\in\Phi$. We by the partial differential equations in Lemma \ref{xxsec4.1} get
$$[y^{p^r}_{-\gamma}, w_d]_{\circ}= \frac{\partial w_d}{\partial y_{\gamma}^{p^s}} [y_{-\gamma}^{p^s}, y_{\gamma}^{p^s}]_{\circ}+\text{other terms}.$$
Note that $[y^{p^r}_{-\gamma}, y^{p^s}_{\gamma}]_{\circ}\neq0$ by Proposition \ref{xxsec3.1}. Also note that
\begin{equation}[y^{p^r}_{-\gamma}, w_d]_{\circ}=w_m\cdot(D_{-\gamma}(r))_{\circ}.  \end{equation}So
\begin{align}
w_m\cdot(D_{-\gamma}(r))_{\circ}&=[y^{p^r}_{-\gamma}, w_d]_{\circ}=\frac{\partial w_d}{\partial y^{p^s}_{\gamma}} [y_{-\gamma}^{p^r}, y_{\gamma}^{p^s}]_{\circ}+\text{other terms}\notag\\& =\frac{\partial w_d}{\partial y^{p^s}_{\gamma}} (\sum_{i=1}^{\ell}n_iy_{\delta_i}^{p^{r+s+1}})+\text{other terms}.
\end{align}
 Let us write
\begin{equation}
(D_{\gamma}(r))_{\circ}=\sum_{i=1}^{\ell}U^{\delta_i}(\underline{y})y^{p^r}_{\delta_{i}}+\text{other terms}.
\end{equation}
Combining (6.2) with (6.3) gives
\begin{equation}
  w_m\,|\,\frac{\partial w_d}{\partial y^{p^s}_{\gamma}} y_{\delta_i}^{p^{r+s+1}-p^r}.
  \end{equation}
  Now we use again the partial differential equations in Lemma \ref{xxsec4.1}. 
  Pick any $\delta\in\Pi$,
   \begin{align}
[y^{p^r}_{\delta}, w_d]_{\circ}&=\frac{\partial w_d}{\partial y^{p^s}_{\gamma}} [y_{\delta}^{p^r}, y_{\gamma}^{p^s}]_{\circ}+\text{other terms}\notag\\& =\frac{\partial w_d}{\partial y^{p^s}_{\gamma}}\langle\gamma, \delta\rangle y_{\gamma}^{p^{r+s+1}}+\text{other terms}.
\end{align}
We also have
 \begin{equation}[y^{p^r}_{\delta}, w_d]_{\circ}=w_m\cdot(D_{\delta}(r))_{\circ}.  \end{equation}
 Similarly writing $(D_{\delta}(r))_{\circ}=U^{\gamma}(\underline{y})y^{p^r}_{\gamma}+\text{other terms}$ and using (6.5) and (6.6) we deduce that
\begin{equation}
  w_m\,|\,\frac{\partial w_d}{\partial y^{p^s}_{\gamma}} y_{\gamma}^{p^{r+s+1}}.
  \end{equation}
  Notice that from (6.4) we have $ w_m\,|\, \frac{\partial w_d}{\partial y^{p^s}_{\gamma}} y_{\delta_i}^{p^{r+s+1}-p^s}$. So,
  $ w_m\,|\,\frac{\partial w_d}{\partial y^{p^s}_{\gamma}}$. This completes the case if $\gamma\in\Phi$.

  The case when $\gamma\in \Pi$ is much more tricky and its proof will use diagram chasing using 
  Dynkin diagram of the Lie type of the group $G(1)$. Given $\Phi$ and $\Pi$, we will
  denote $\alpha_{\max}$ to be the highest root. The philosophy of the proof is to write the partial 
  differential equations in Lemma \ref{xxsec4.1} for all $y_{\alpha_i}$, where
  $\alpha_i$ is a simple root equal to $\delta_i$\,($\alpha_i=\delta_i$). Note that $y_{\alpha_i}$ 
  and $y_{\delta_i}$ are different elements. When we write $y_{\alpha_i}$, we mean the 
  element $y_{\alpha_i}=x_{\alpha_i}(p)-1$ but $y_{\delta_i}=h_{\delta_i}(1+p)-1$.
  So we will write
  $[y_{\alpha_i}, w_d]_{\circ}$, $[y_{-\alpha_i}, w_d]_{\circ}$, $[y_{-\alpha_{\max}}, w_d]_{\circ}$ and $[y_{\alpha_{\max}}, w_d]_{\circ}$. This will give us divisibility relations
  of $w_m$ which we will then solve to show that $ w_m\,|\,\frac{\partial w_d}{\partial y^{p^s}_{\delta_i}} $ for all $ i$. Let us start with the case $B_{\ell}$, which suit us best to explain the philosophy of our diagram chasing argument.

  Dynkin diagram for ${\ell}\geq 3$\,(cf. \cite[Page 214]{Bourbaki}):
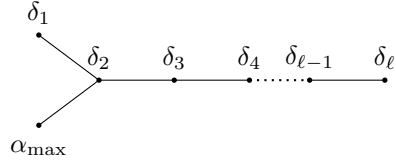
  \begin{center}
\begin{tikzpicture}
\draw(1.2,-0.6)--(2,0);
\fill (1.2,-0.6) circle (1pt);
\node at(1.2,-0.9){$\alpha_{\max}$};
\draw(1.2,0.6)--(2,0);
\fill (1.2,0.6) circle (1pt);
\node at(1.2,0.9){$\delta_1$};
\draw(2,0)--(3,0);
\fill (2,0) circle (1pt);
\fill (3,0) circle (1pt);
\node at(2,0.3){$\delta_2$};
\draw(3,0)--(4,0);
\fill (4,0) circle (1pt);
\node at(3,0.3){$\delta_3$};
\node at(4,0.3){$\delta_4$};
\draw[dotted, thick](4.1,0)--(4.75,0);
\draw(4.8,0)--(5.8,0);
\fill (4.8,0) circle (1pt);
\fill (5.8,0) circle (1pt);
\node at(4.8,0.3){$\delta_{\ell-1}$};
\node at(5.8,0.3){$\delta_{\ell}$};
\end{tikzpicture}
\end{center}

Note that, by \cite[Page 207]{Bourbaki} and the discussion on Dynkin graphs, the fact 
that $\alpha_{\max}$ is linked to $\delta_2$ only means that $\langle\alpha_{\max},\delta_2\rangle \neq0$ and
$\langle\alpha_{\max},\delta_i\rangle =0$ for all $i=1,3,4,\cdots,\ell$. Let us first write the 
partial differential equation for $\alpha_{\max}$, so
\begin{equation}
w_m\cdot(D_{\alpha_{\max}}(r))_{\circ}=[y^{p^r}_{\alpha_{\max}}, w_d]_{\circ}=\frac{\partial w_d}{\partial y^{p^s}_{\delta_2}} [y^{p^r}_{\alpha_{\max}},y^{p^s}_{\delta_2}]_{\circ}
+\text{other terms}.\end{equation}
By expanding  $(D_{\alpha_{\max}}(r))_{\circ}$ and by the same method as that of 
(6.1)-(6.7)\,(noticing that $[y^{p^r}_{\alpha_{\max}}$, $y^{p^s}_{\delta_2}]_{\circ}$$=-\langle\alpha_{\max},\delta_2\rangle y_{\alpha_{\max}}^{p^{r+s+1}} $), 
we obtain $w_m\,|\,\frac{\partial w_d}{\partial y^{p^s}_{\delta_2}} y_{\alpha_{\max}}^{p^{r+s+1}-p^r}$. 
Replacing $y_{\alpha_{max}}$ in (6.8) by $y_{-\alpha_{max}}$, we get $w_m\,|\,\frac{\partial w_d}{\partial y^{p^s}_{\delta_2}} y_{-\alpha_{\max}}^{p^{r+s+1}-p^r}$. 
This implies that $w_m\,|\,\frac{\partial w_d}{\partial y^{p^s}_{\delta_2}}$.

Recall the Dynkin diagram above 
  
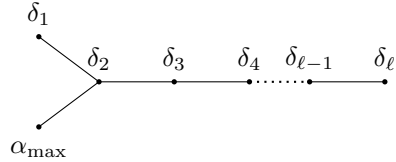
\begin{center}
\begin{tikzpicture}
\draw(1.2,-0.6)--(2,0);
\fill (1.2,-0.6) circle (1pt);
\node at(1.2,-0.9){$\alpha_{\max}$};
\draw(1.2,0.6)--(2,0);
\fill (1.2,0.6) circle (1pt);
\node at(1.2,0.9){$\delta_1$};
\draw(2,0)--(3,0);
\fill (2,0) circle (1pt);
\fill (3,0) circle (1pt);
\node at(2,0.3){$\delta_2$};
\draw(3,0)--(4,0);
\fill (4,0) circle (1pt);
\node at(3,0.3){$\delta_3$};
\node at(4,0.3){$\delta_4$};
\draw[dotted, thick](4.1,0)--(4.75,0);
\draw(4.8,0)--(5.8,0);
\fill (4.8,0) circle (1pt);
\fill (5.8,0) circle (1pt);
\node at(4.8,0.3){$\delta_{\ell-1}$};
\node at(5.8,0.3){$\delta_{\ell}$};
\end{tikzpicture}
\end{center}
Notice that writing the partial differential equations 
for $[y^{p^r}_{\alpha_{\max}}, w_d]_{\circ}$ and $[y^{p^r}_{-\alpha_{\max}}, w_d]_{\circ}$ gives 
us $w_m\,|\,\frac{\partial w_d}{\partial y^{p^s}_{\delta_2}}$\,(the main fact is  that $\alpha_{\max}$ 
is only linked with $\delta_2$).

Now write the partial differential equations $[y^{p^r}_{\alpha_1}, w_d]_{\circ}$ and $[y^{p^r}_{-\alpha_1}, w_d]_{\circ}$, 
where $\alpha_1=\delta_1$,
$y_{\alpha_1}=x_{\alpha_1}(p)-1$ is not the same as $y_{\delta_1}=h_{\delta_1}(1+p)-1$. 
Since $\delta_1$ is only connected with $\delta_2$, we will arrive at  $w_m\,|\,\frac{\partial w_d}{\partial y^{p^s}_{\delta_1}}$. 
We clarify this below for the convenience of the reader.
 \begin{align*}w_m\cdot(D_{\alpha_1}(r))_{\circ}&=[y^{p^r}_{\alpha_1}, w_d]_{\circ}=\frac{\partial w_d}{\partial y^{p^s}_{\delta_1}} [y^{p^r}_{\alpha_1},y^{p^s}_{\delta_1}]_{\circ}+
\frac{\partial w_d}{\partial y^{p^s}_{\delta_2}} [y^{p^r}_{\alpha_1},y^{p^s}_{\delta_2}]_{\circ}+\text{other terms}\\
&=-\frac{\partial w_d}{\partial y^{p^s}_{\delta_1}} \langle\alpha_1,\delta_1\rangle y_{\alpha_1}^{p^{r+s+1}}-\langle\alpha_1,\delta_2\rangle\frac{\partial w_d}{\partial y^{p^s}_{\delta_2}} y_{\alpha_1}^{p^{r+s+1}}
+\text{other terms}. \end{align*}
Thus we get $$w_m\,|\,(\frac{\partial w_d}{\partial y^{p^s}_{\delta_1}}\langle\alpha_1,\delta_1\rangle +\frac{\partial w_d}{\partial y^{p^s}_{\delta_2}}\langle\alpha_1,\delta_2\rangle  ) y_{\alpha_1}^{p^{r+s+1}}.$$
Replacing $\alpha_1$ by $-\alpha_1$ yields 
$$w_m\,|\,(\frac{\partial w_d}{\partial y^{p^s}_{\delta_1}}\langle\alpha_1,\delta_1\rangle +\frac{\partial w_d}{\partial y^{p^s}_{\delta_2}}\langle\alpha_1,\delta_2\rangle  ) y_{-\alpha_1}^{p^{r+s+1}}.$$
As $\gcd (y_{\alpha_1},y_{-\alpha_1})=1$, we have
$$w_m\,|\,(\frac{\partial w_d}{\partial y^{p^s}_{\delta_1}}\langle\alpha_1,\delta_1\rangle +\frac{\partial w_d}{\partial y^{p^s}_{\delta_2}}\langle\alpha_1,\delta_2\rangle  ) .$$
Now we have already shown that  $w_m\,|\,\frac{\partial w_d}{\partial y^{p^s}_{\delta_2}}$. This gives that  $w_m\,|\,\frac{\partial w_d}{\partial y^{p^s}_{\delta_1}}$ as $\langle\alpha_1,\delta_1\rangle =\langle\delta_1,\delta_1\rangle\neq0 $\,(by \cite{{Bourbaki}}). We proceed this diagram chasing for other simple roots in the Dynkin diagram of $B_{\ell}$.

Now writing $[y^{p^r}_{\alpha_2}, w_d]_{\circ}$  and  $[y^{p^r}_{-\alpha_2}, w_d]_{\circ}$ for $\alpha_2=\delta_2$, 
we get
$$
w_m\,|\, (\frac{\partial w_d}{\partial y^{p^s}_{\delta_1}}\langle\delta_2,\delta_1\rangle +\frac{\partial w_d}{\partial y^{p^s}_{\delta_2}}\langle\delta_2,\delta_2\rangle
+\frac{\partial w_d}{\partial y^{p^s}_{\delta_3}}\langle\delta_2,\delta_3\rangle),
$$
which is due to the fact that $\delta_2$ is only connected with $\delta_1$ and $\delta_3$. Since 
we have already shown that $w_m|\frac{\partial w_d}{\partial y^{p^s}_{\delta_1}}$ and 
$w_m|\frac{\partial w_d}{\partial y^{p^s}_{\delta_2}}$, we will obtain $w_m|\frac{\partial w_d}{\partial y^{p^s}_{\delta_3}}$. 
We proceed this to complete and finally we conclude that
$w_m\,|\, \frac{\partial w_d}{\partial y^{p^s}_{\delta_i}}$, $i\in[1,\ell]$. The above 
argument also works for types
 
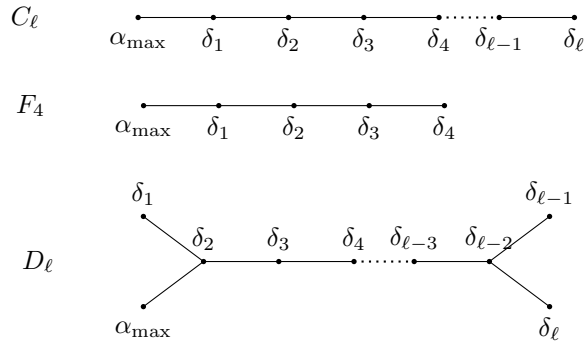
\begin{center}
\begin{tikzpicture}
\draw(0,0)--(1,0);
\fill (0,0) circle (1pt);
\node at(0,-0.3){$\alpha_{\max}$};
\node at(-1.5,0){$C_{\ell}$};
\draw(1,0)--(2,0);
\fill (1,0) circle (1pt);
\node at(1.0,-0.3){$\delta_1$};
\draw(2,0)--(3,0);
\fill (2,0) circle (1pt);
\fill (3,0) circle (1pt);
\node at(2,-0.3){$\delta_2$};
\draw(3,0)--(4,0);
\fill (4,0) circle (1pt);
\node at(3,-0.3){$\delta_3$};
\node at(4,-0.3){$\delta_4$};
\draw[dotted, thick](4.1,0)--(4.75,0);
\draw(4.8,0)--(5.8,0);
\fill (4.8,0) circle (1pt);
\fill (5.8,0) circle (1pt);
\node at(4.8,-0.3){$\delta_{\ell-1}$};
\node at(5.8,-0.3){$\delta_{\ell}$};
\end{tikzpicture}
\end{center}
\begin{center}
\begin{tikzpicture}
\draw(0,0)--(1,0);
\fill (0,0) circle (1pt);
\node at(0,-0.3){$\alpha_{\max}$};
\node at(-1.5,0){$F_4$};
\draw(1,0)--(2,0);
\fill (1,0) circle (1pt);
\node at(1.0,-0.3){$\delta_1$};
\draw(2,0)--(3,0);
\fill (2,0) circle (1pt);
\fill (3,0) circle (1pt);
\node at(2,-0.3){$\delta_2$};
\draw(3,0)--(4,0);
\fill (4,0) circle (1pt);
\node at(3,-0.3){$\delta_3$};
\node at(4,-0.3){$\delta_4$};
\node at(5.8,-0.3){ };
\end{tikzpicture}
\end{center}
\begin{center}
\begin{tikzpicture}
\node at(-0.2,0){$D_{\ell}$};
\draw(1.2,-0.6)--(2,0);
\fill (1.2,-0.6) circle (1pt);
\node at(1.2,-0.9){$\alpha_{\max}$};
\draw(1.2,0.6)--(2,0);
\fill (1.2,0.6) circle (1pt);
\node at(1.2,0.9){$\delta_1$};
\draw(2,0)--(3,0);
\fill (2,0) circle (1pt);
\fill (3,0) circle (1pt);
\node at(2,0.3){$\delta_2$};
\draw(3,0)--(4,0);
\fill (4,0) circle (1pt);
\node at(3,0.3){$\delta_3$};
\node at(4,0.3){$\delta_4$};
\draw[dotted, thick](4.1,0)--(4.75,0);
\draw(4.8,0)--(5.8,0);
\fill (4.8,0) circle (1pt);
\fill (5.8,0) circle (1pt);
\node at(4.8,0.3){$\delta_{\ell-3}$};
\node at(5.8,0.3){$\delta_{\ell-2}$};
\draw(5.8,0)--(6.6,0.6);
\fill (6.6,0.6) circle (1pt);
\node at(6.6,-0.9){$\delta_{\ell}$};
\draw(5.8,0)--(6.6,-0.6);
\fill (6.6,-0.6) circle (1pt);
\node at(6.6,0.9){$\delta_{\ell-1}$};
\end{tikzpicture}
\end{center}

For $D_{\ell}$, just as in the case for $B_{\ell}$, we start from $\alpha_{\max}$ and do diagram chasing until we get $w_m\,|\, \frac{\partial w_d}{\partial y^{p^s}_{\delta_{\ell-2}}}$ from the equations $$w_m\cdot(D_{\alpha_{\ell-3}}(r))_{\circ}=[y^{p^r}_{\alpha_{\ell-3}}, w_d]_{\circ}$$
and
$$w_m\cdot(D_{-\alpha_{\ell-3}}(r))_{\circ}=[y^{p^r}_{-\alpha_{\ell-3}}, w_d]_{\circ},\  \alpha_{\ell-3}=\delta_{\ell-3}.$$
Finally, by invoking the relations
$$w_m\cdot(D_{\alpha_{\ell}}(r))_{\circ}=[y^{p^r}_{\alpha_{\ell}}, w_d]_{\circ} \ \text{and}\
w_m\cdot(D_{-\alpha_{\ell}}(r))_{\circ}=[y^{p^r}_{-\alpha_{\ell}}, w_d]_{\circ},$$
we arrive at $w_m\,|\, \frac{\partial w_d}{\partial y^{p^s}_{\delta_{\ell}}}$ . Similarly, using 
the relations
$$
w_m\cdot(D_{\alpha_{\ell-1}}(r))_{\circ}=[y^{p^r}_{\alpha_{\ell-1}}, w_d]_{\circ} \ \text{and}\
w_m\cdot(D_{-\alpha_{\ell-1}}(r))_{\circ}=[y^{p^r}_{-\alpha_{\ell-1}}, w_d]_{\circ},
$$ 
we obtain
 $w_m\,|\, \frac{\partial w_d}{\partial y^{p^s}_{\delta_{\ell-1}}}$.

Here is the Dynkin diagram for $E_8$.
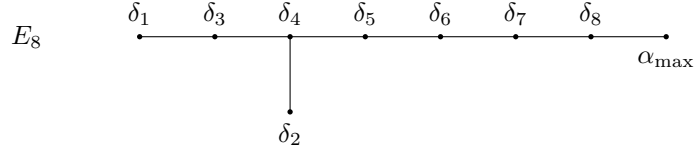
\begin{center}
\begin{tikzpicture}
\draw(0,0)--(1,0);
\fill (0,0) circle (1pt);
\node at(0,0.3){$\delta_1$};
\node at(-1.5,0){$E_8$};
\draw(1,0)--(2,0);
\fill (1,0) circle (1pt);
\node at(1.0,0.3){$\delta_3$};
\draw(2,0)--(3,0);
\fill (2,0) circle (1pt);
\fill (3,0) circle (1pt);
\node at(2,0.3){$\delta_4$};
\draw(2,0)--(2,-1);
\fill (2,-1) circle (1pt);
\node at(2,-1.3){$\delta_2$};
\draw(3,0)--(4,0);
\fill (4,0) circle (1pt);
\node at(3,0.3){$\delta_5$};
\node at(4,0.3){$\delta_6$};
\draw(4,0)--(5,0);
\fill (5,0) circle (1pt);
\node at(5,0.3){$\delta_7$};
\draw(5,0)--(6,0);
\fill (6,0) circle (1pt);
\node at(6,0.3){$\delta_8$};
\draw(6,0)--(7,0);
\fill (7,0) circle (1pt);
\node at(7,-0.3){$\alpha_{\max}$};
\end{tikzpicture}
\end{center}
Case $E_8$ is also easy to see, we adopt the same trick as we did for the case of $D_{\ell}$.

Here is the Dynkin diagram for $E_7$.
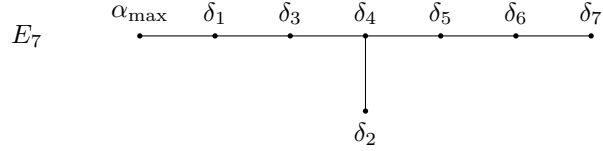
\begin{center}
\begin{tikzpicture}
\draw(0,0)--(1,0);
\fill (0,0) circle (1pt);
\node at(0,0.3){$\alpha_{\max}$};
\node at(-1.5,0){$E_7$};
\draw(1,0)--(2,0);
\fill (1,0) circle (1pt);
\node at(1.0,0.3){$\delta_1$};
\draw(2,0)--(3,0);
\fill (2,0) circle (1pt);
\fill (3,0) circle (1pt);
\node at(2,0.3){$\delta_3$};
\draw(3,0)--(3,-1);
\fill (3,-1) circle (1pt);
\node at(3,-1.3){$\delta_2$};
\draw(3,0)--(4,0);
\fill (4,0) circle (1pt);
\node at(3,0.3){$\delta_4$};
\node at(4,0.3){$\delta_5$};
\draw(4,0)--(5,0);
\fill (5,0) circle (1pt);
\node at(5,0.3){$\delta_6$};
\draw(5,0)--(6,0);
\fill (6,0) circle (1pt);
\node at(6,0.3){$\delta_7$};
\end{tikzpicture}
\end{center}
Case $E_7$ is also similar to $E_8$ and $D_{\ell}$ and we do the same diagram chasing by using the partial differential equations as we did earlier.

Here is the Dynkin diagram for $E_6$.
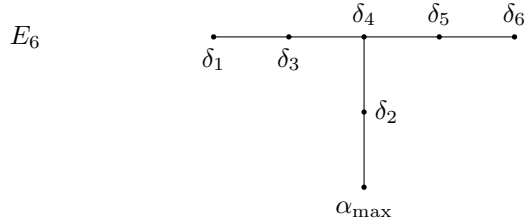
\begin{center}
\begin{tikzpicture}
\node at(-1.5,0){$E_6$};
\draw(1,0)--(2,0);
\fill (1,0) circle (1pt);
\node at(1.0,-0.3){$\delta_1$};
\draw(2,0)--(3,0);
\fill (2,0) circle (1pt);
\fill (3,0) circle (1pt);
\node at(2,-0.3){$\delta_3$};
\draw(3,0)--(3,-1);
\fill (3,-1) circle (1pt);
\node at(3.3,-1){$\delta_2$};
\draw(3,-1)--(3,-2);
\fill (3,-2) circle (1pt);
\node at(3,-2.3){$\alpha_{\max}$};
\draw(3,0)--(4,0);
\fill (4,0) circle (1pt);
\node at(3,0.3){$\delta_4$};
\node at(4,0.3){$\delta_5$};
\draw(4,0)--(5,0);
\fill (5,0) circle (1pt);
\node at(5,0.3){$\delta_6$};
\end{tikzpicture}
\end{center}
This $E_6$ case is a little tricky and we will demonstrate it in a detailed way. In 
view of the facts
$$w_m\cdot(D_{\alpha_{\max}}(r))_{\circ}=[y^{p^r}_{\alpha_{\max}}, w_d]_{\circ} 
$$
and
$$
w_m\cdot(D_{-\alpha_{\max}}(r))_{\circ}=[y^{p^r}_{-\alpha_{\max}}, w_d]_{\circ},
$$
we observe that $w_m\,|\, \frac{\partial w_d}{\partial y^{p^s}_{\delta_2}}$. Taking
into account the relations
$$
w_m\cdot(D_{\alpha_2}(r))_{\circ}=[y^{p^r}_{\alpha_2}, w_d]_{\circ}
$$
and 
$$
w_m\cdot(D_{-\alpha_2}(r))_{\circ}=[y^{p^r}_{-\alpha_2}, w_d]_{\circ},
$$
we obtain $w_m\,|\, \frac{\partial w_d}{\partial y^{p^s}_{\delta_4}}$\,(because $\delta_2$ is linked with $\delta_4$).
Now we write, for $\alpha_3=\delta_3$, 
$$
w_m\cdot(D_{\alpha_3}(r))_{\circ}=[y^{p^r}_{\alpha_3}, w_d]_{\circ} \ \text{and}\
w_m\cdot(D_{-\alpha_3}(r))_{\circ}=[y^{p^r}_{-\alpha_3}, w_d]_{\circ},$$
and these will give us
$$w_m\,|\, (\frac{\partial w_d}{\partial y^{p^s}_{\delta_1}}\langle\delta_3,\delta_1\rangle +\frac{\partial w_d}{\partial y^{p^s}_{\delta_3}}\langle\delta_3,\delta_3\rangle
+\frac{\partial w_d}{\partial y^{p^s}_{\delta_4}}\langle\delta_3,\delta_4\rangle   ) .$$
Applying the fact $w_m\,|\, \frac{\partial w_d}{\partial y^{p^s}_{\delta_4}}$ yields that 
\begin{equation}w_m\,|\, (\frac{\partial w_d}{\partial y^{p^s}_{\delta_1}}\langle\delta_3,\delta_1\rangle +2\frac{\partial w_d}{\partial y^{p^s}_{\delta_3}}).\end{equation}
It follows from the relations
 $$w_m\cdot(D_{\alpha_1}(r))_{\circ}=[y^{p^r}_{\alpha_1}, w_d]_{\circ} \ \text{and}\
w_m\cdot(D_{-\alpha_1}(r))_{\circ}=[y^{p^r}_{-\alpha_1}, w_d]_{\circ}
$$ 
that
$$w_m\,|\, (\frac{\partial w_d}{\partial y^{p^s}_{\delta_1}}\langle\delta_1,\delta_1\rangle 
+\frac{\partial w_d}{\partial y^{p^s}_{\delta_3}}\langle\delta_1,\delta_3\rangle).$$
This implies that
\begin{equation}
w_m\,|\, (2\frac{\partial w_d}{\partial y^{p^s}_{\delta_1}} +\frac{\partial w_d}{\partial y^{p^s}_{\delta_3}}\langle\delta_1,\delta_3\rangle).
\end{equation}
Now by \cite[Page 5]{Steinberg}, we know that
$$\langle\delta_3,\delta_1\rangle\in\{\pm 1\},\
\langle\delta_1,\delta_3\rangle
=2\frac{(\delta_1\,|\, \delta_3)}{(\delta_3\,|\,\delta_3)}=(\delta_1\,|\, \delta_3)\ \text{and}\
(x\,|\, y)=(y\,|\, x).
$$
So $\langle\delta_1,\delta_3\rangle=\langle\delta_3,\delta_1\rangle$. 
We therefore assume that $\langle\delta_1,\delta_3\rangle=+1$. By the relation (6.9) we see that
$$
w_m\,|\, (2\frac{\partial w_d}{\partial y^{p^s}_{\delta_1}}+\frac{\partial w_d}{\partial y^{p^s}_{\delta_3}}).
$$
Let us set $x=\frac{\partial w_d}{\partial y^{p^s}_{\delta_3}}$ and $y=\frac{\partial w_d}{\partial y^{p^s}_{\delta_1}}$. 
Thus $w_m\,|\, (2x+y)$ and $w_m\,|\, (2y+x)$. And hence
$w_m\,|\, [2(2x+y)-(2y+x)=3x]$. Therefore, 
$w_m\,|\, x$ and then from (6.10) we get  $w_m\,|\, y$. So  $w_m\,|\, \frac{\partial w_d}{\partial y^{p^s}_{\delta_1}}$ 
and  $w_m\,|\, \frac{\partial w_d}{\partial y^{p^s}_{\delta_3}}$. Similarly, using the  equations
\begin{align*}w_m\cdot(D_{\alpha_5}(r))_{\circ}&=[y^{p^r}_{\alpha_5}, w_d]_{\circ}\,(\alpha_5=\delta_5) ,\\
w_m\cdot(D_{\alpha_6}(r))_{\circ}&=[y^{p^r}_{\alpha_6}, w_d]_{\circ},\,(\alpha_6=\delta_6),\end{align*}
and solving them we get $w_m\,|\, \frac{\partial w_d}{\partial y^{p^s}_{\delta_5}}$ and  $w_m\,|\, \frac{\partial w_d}{\partial y^{p^s}_{\delta_6}}$. 
This completes the proof of case $E_6$.

Here is the Dynkin diagram for $G_2$.
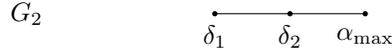
\begin{center}
\begin{tikzpicture}
\node at(-1.5,0){$G_2$};
\draw(1,0)--(2,0);
\fill (1,0) circle (1pt);
\node at(1.0,-0.3){$\delta_1$};
\draw(2,0)--(3,0);
\fill (2,0) circle (1pt);
\fill (3,0) circle (1pt);
\node at(2,-0.3){$\delta_2$};
\node at(3,-0.3){$\alpha_{\max}$};

\end{tikzpicture}
\end{center}
This case is easy and similar to the cases of $B_{\ell}$ and $D_{\ell}$. 

Let us next deal with the case of $A_{\ell}$. 
The Dynkin diagram for $A_{\ell}$ is  the following
 
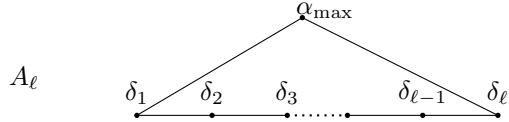
\begin{center}
\begin{tikzpicture}
\node at(0.5,0.5){$A_{\ell}$};
\draw(4.2,1.3)--(2,0);
\fill (4.2,1.3) circle (1pt);
\node at(4.5,1.4){$\alpha_{\max}$};
\draw(4.2,1.3)--(6.8,0);
\draw(2,0)--(3,0);
\fill (2,0) circle (1pt);
\fill (3,0) circle (1pt);
\node at(2,0.3){$\delta_1$};
\draw(3,0)--(4,0);
\fill (4,0) circle (1pt);
\node at(3,0.3){$\delta_2$};
\node at(4,0.3){$\delta_3$};
\draw[dotted, thick](4.1,0)--(4.75,0);
\draw(4.8,0)--(5.8,0);
\draw(5.8,0)--(6.8,0);
\fill (4.8,0) circle (1pt);
\fill (5.8,0) circle (1pt);
\fill (6.8,0) circle (1pt);
\node at(5.8,0.3){$\delta_{\ell-1}$};
\node at(6.8,0.3){$\delta_{\ell}$};
\end{tikzpicture}
\end{center}
Let us write $x_i=\frac{\partial w_d}{\partial y^{p^s}_{\delta_i}}$, $\alpha_i=\delta_i$. Then 
the following two relations 
 $$w_m\cdot(D_{\alpha_1}(r))_{\circ}=[y^{p^r}_{\alpha_1}, w_d]_{\circ} \ \text{and}\
w_m\cdot(D_{-\alpha_1}(r))_{\circ}=[y^{p^r}_{-\alpha_1}, w_d]_{\circ}$$ give
$$w_m\,|\, (-\langle\alpha_1,\delta_1\rangle \frac{\partial w_d}{\partial y^{p^s}_{\delta_1}}-\langle\alpha_1,\delta_2\rangle \frac{\partial w_d}{\partial y^{p^s}_{\delta_2}} ),$$
which is $w_m\,|\, (-2 \frac{\partial w_d}{\partial y^{p^s}_{\delta_1}}+\frac{\partial w_d}{\partial y^{p^s}_{\delta_2}})$ due to \cite[Page 217]{Bourbaki}. 
Thus we have 
\begin{equation}
w_m\,|\, (-2x_1+x_2).
\end{equation}
Similarly, writing
$$w_m\cdot(D_{\alpha_2}(r))_{\circ}=[y^{p^r}_{\alpha_2}, w_d]_{\circ} \ \text{and}\
w_m\cdot(D_{-\alpha_2}(r))_{\circ}=[y^{p^r}_{-\alpha_2}, w_d]_{\circ},$$
we get
$$w_m\,|\, (-\langle\alpha_2,\alpha_1\rangle\frac{\partial w_d}{\partial y^{p^s}_{\delta_1}}-\langle\alpha_2,\alpha_2\rangle\frac{\partial w_d}{\partial y^{p^s}_{\delta_2}}
-\langle\alpha_2,\alpha_3\rangle\frac{\partial w_d}{\partial y^{p^s}_{\delta_3}} ) ,$$
ie. \begin{equation}
w_m\,|\, (x_1-2x_2+x_3).
\end{equation}
Writing the other partial differential equations, we finally arrive at
\begin{align*}
&w_m\,|\, (-2x_1+x_2),\\
&w_m\,|\, (x_1-2x_2+x_3),\\
&w_m\,|\, (x_2-2x_3+x_4)=C_2\,(\text{say}),\\
&w_m\,|\, (x_3-2x_4+x_5)=C_3,\\
& \ \ \ \ \ \ \ \ \ \ \vdots\\
&w_m\,|\, (x_{\ell-2}-2x_{\ell-1}+x_{\ell} =C_{\ell-2}),\\
&w_m\,|\, (x_{\ell-1}-2x_{\ell}=C_{\ell-1}),\\
&w_m\,|\, (x_1+x_{\ell} =C_{\ell})\,(\text{this is obtained by using}\, \, a_{max} ).\\
\end{align*}
Now note that there exist constants $d_1,d_2, \cdots, d_{\ell-1}$ such that
$\sum_{i=2}^{\ell-1}d_iC_i+d_1C_1$ is of the form $ax_1+x_2$, where $d_i$ for $i\in[1,\ell-1]$ and $a$ are positive constants.  
We therefore conclude that there exists a positive constant $a$ such that $w_m\,|\, (ax_1+x_2)$. Now, 
$w_m\,|\, (-2x_1+x_2))$, and so $w_m\,|\, [(a+2)x_1=(ax_1+x_2)-(-2x_1+x_2)]$. This shows that $w_m\,|\, x_1$. 
Now recursively, using the relations after (6.12), we eventually arrive at $w_m\,|\, x_i$ for all $i$. This completes the proof of Claim \ref{xxsec5.3}.
\end{proof}

\subsection{Applications to center}\label{xxsec7}

Recall that $G$ is a semi-simple, simply connected, split Chevalley group 
over $\mathbb{Z}_p$, $G(1)$ is the first congruence kernel of $G$ and $\Omega_{G(1)}$ 
is the mod-$p$ Iwasawa algebra of $G(1)$ over $\mathbb{F}_p$.

As a direct consequence of Theorem \ref{xxsec5.1}, we have

\begin{proposition}\label{xxsec7.1}
Let $p$ be a prime with $p\geq 3$ and $G$ be a semi-simple, simply connected, split Chevalley group 
over $\mathbb{Z}_p$. Suppose that $G$ is one of the following Chevalley groups of Lie type: 
$A_{\ell}\, (\ell\geq 1), B_{\ell}\,(\ell\geq 2), C_{\ell}\,(\ell\geq 2), D_{\ell}\,(\ell\geq 3), E_6, E_7, E_8, F_4, G_2$.
Then the center of $\Omega_{G(1)}$ is trivial, i.e. if $r$ is central element of $\Omega_{G(1)}$, then $r\in \mathbb{F}_p$.
\end{proposition}

\begin{proof}
Suppose that $r$ is a central element of $\Omega_{G(1)}$. Then by Theorem \ref{xxsec5.1}, we can write
$$r=r_0+r_1 + r_2 + \cdots+ r_d + \cdots,
$$
where $r_d\, (d=0, 1, 2, \cdots)$ are homogeneous polynomials with respect
to $y_{12}$, $y_{13}$, $\cdots$, $y_{1122}$, $\cdots$, $y_{n(n-1)}$ of degree $d$.

Since $r$ is a central element, we have $[x,r]=0 $
for all $x\in\Omega_{G(1)}$. This implies
$$
[x,r_1 + r_2 + \cdots+ r_d + \cdots]=0
$$
for all $x\in\Omega_{G(1)}$.  By  Theorem \ref{xxsec5.1} again, we assert that
$$
r_1 + r_2 + \cdots+ r_d + \cdots=0,
$$
and the result follows.
\end{proof}

We therefore say that the center of $\Omega_{G(1)}$ is exactly the finite field 
$\mathbb{F}_p$. This accounts to reproving Ardakov's result \cite[Corollary A]{Ardakov}.

\subsection{Future questions}
\label{applications nonuniform}

Clozel \cite{Clozel2} considered the Iwasawa algebra of the pro-$p$ 
Iwahori subgroup of ${\rm GL}_2(L)$ for an unramified extension $L$ of degree $r$ of $\mathbb{Q}_p$ and 
gave a presentation of it by generators and relations. Inspired by Clozel's systematic works, Ray \cite{Ray3}
extend his result to determine the explicit ring-theoretic presentation, in the form of generators and relations, of the 
Iwasawa algebra of the pro-$p$ Iwahori subgroup of ${\rm GL}_n(\mathbb{Z}_p)$.

Clozel's and Ray's works show that there is considerable interest towards understanding 
the structure of Iwasawa algebras over the pro-$p$ Iwahori subgroup $G$ of ${\rm GL}_n(\mathbb{Z}_p)$. 
Moreover, Bushnell and Henniart \cite[Chapter 4, Section 17]{BushnellHenniart} together with Herzig \cite[Lemma 10]{Herzig} 
have pointed out that for the pro-$p$ Iwahori subgroup $G$ of ${\rm GL}_2(\mathbb{Q}_p)$ and any nonzero 
irreducible (resp. smooth) representation $V$,  the class of irreducible (resp. smooth) representation $(\pi, V)$ of ${\rm GL}_2(\mathbb{Q}_p)$ for which 
the set of $G$-fixed vectors $V^G\neq 0$ is particularly subtle and useful in mod-$p$ 
representation theory of $p$-adic groups. Furthermore, the pro-$p$ Iwahori and its associated Hecke algebra have several applications in the emerging Langlands program (see the works of M. F. Vigneras). Therefore a natural question is to understand the pro-$p$ Iwahori subgroups and their associated Iwasawa algebra.

Let $L/\mathbb{Q}_p$ be a finite extension and $e$ its ramification index. Suppose that $L$ 
is \textit{mildly ramified}, i.e.
$$
ne<p-1,
$$
see \cite[Chapter III, 3.2]{Lazard}. Let $\mathfrak{p}$, $\mathcal{O}= \mathcal{O}_L$ denote 
the prime ideal, the integers of $L$, respectively. 

We denote by $G$ the pro-$p$ Iwahori subgroup of ${\rm GL}_2(\mathbb{Z}_p)$, i.e.
$$
G=\left\{ g\in {\rm GL}_2(\mathbb{Z}_p) \  \vline \  g \equiv  \left[
\begin{array}
[c]{cc}%
1 & \ast \\
0 & 1 \\
\end{array}
\right]\ [p]  \right\}.
$$

\begin{proposition}{\rm \cite[Chapter III (3.2.7)]{Lazard}}\label{xxsec7.3}
Let $p>3$, then the pro-$p$ Iwahori subgroup $G$ of ${\rm GL}_2(\mathbb{Z}_p)$ is a $p$-valued $p$-saturated Sylow 
subgroup in the sense of Lazard.
\end{proposition}

Applying the arguments of \cite[Section 2]{Clozel2} to ${\rm GL}_2(\mathbb{Z}_p)$, we see that 
\begin{align*}
x=\left [
\begin{array}{cc}
1 & 1\\
0 & 1
\end{array}\right ], \quad
h=\left [
\begin{array}{cc}
1+p& 0\\
0 & (1+p)^{-1}
\end{array}\right ], \quad
y=\left [
\begin{array}{cc}
1 & 0\\
p & 1
\end{array}\right ]
\end{align*}
form a topological generating set for the pro-$p$ Iwahori subgroup $H$ of
 ${\rm SL}_2(\mathbb{Z}_p) $ and that \begin{align*}
& x=\left [
\begin{array}{cc}
1 & 1\\
0 & 1
\end{array}\right ], \quad
 h=\left [
\begin{array}{cc}
1+p& 0\\
0 & (1+p)^{-1}
\end{array}\right ], \quad\\&
z=\left [
\begin{array}{cc}
1+p& 0\\
0 & 1+p
\end{array}\right ], \quad
y=\left [
\begin{array}{cc}
1 & 0\\
p & 1
\end{array}\right ]
\end{align*}
construct a topological generating set for the pro-$p$ Iwahori subgroup $G$ of
 ${\rm GL}_2(\mathbb{Z}_p) $. When certain complicated computations are needed, 
 the type and the number of topological generators will be useful.  Let us set
 $$
 {\bf x}=x-1,\,  {\bf h}=h-1, \, {\bf z}=z-1,\,  {\bf y}=y-1.
 $$
Then ${\bf x},\,  {\bf h},\, {\bf y}\in
{\mathbb{F}}_p[H]\subseteq \Omega_H$ and ${\bf x},\,  {\bf h}, \, {\bf z},\, {\bf y}\in
{\mathbb{F}}_p[G]\subseteq \Omega_G$. Thus we can produce various
monomials in the ${\bf x},\, {\bf h},\,  {\bf z}, \, {\bf y}$: if
$\alpha=(i, j, k)$ is a 3-tuple of nonnegative integers, we define
$$
{\rm {\bf X}}^\alpha={\bf x}^i {\bf h}^j {\bf y}^k\in
\Omega_H
$$
If  $\beta=(i, j, k, l)$ is a 4-tuple of nonnegative integers, 
one can define
$$
{\rm {\bf Y}}^\beta={\bf x}^i {\bf h}^j {\bf z}^k{\bf y}^l\in
\Omega_G
$$

It should be remarked that the expressions of these monomials depend
on our choice of ordering of the ${\bf x}$'s,
${\bf h}$'s, ${\bf z}$'s, ${\bf y}$'s, because $\Omega_H$ and $\Omega_G$ are 
noncommutative unless $H$ and $G$ are abelian. The following result shows
that $\Omega_H$ and $\Omega_G$ are both ``noncommutative formal power series rings".

\begin{proposition}\label{xxsec2.5}{\rm (cf. \cite[Chapter VI, Section 28]{Schneider})}
Every element $a$ of the mod-$p$ Iwasawa algebra $\Omega_H$ can be 
written as the sum of a uniquely determined convergent series
$$
a=\sum_{\alpha \in {\mathbb{N}}^3} a_\alpha {\rm
{\bf{X}}}^\alpha,
$$
where $a_\alpha\in {\mathbb{F}}_p$ for all $\alpha\in
{\mathbb{N}}^3$; each element $b$ of the mod-$p$ Iwasawa algebra $\Omega_G$ is equal
to the sum of a uniquely determined convergent series
$$
b=\sum_{\beta \in {\mathbb{N}}^4} b_\beta {\rm
{\bf{Y}}}^\beta, 
$$
where $b_\beta\in {\mathbb{F}}_p$ for all $\beta\in
{\mathbb{N}}^4$.
\end{proposition}
One can analogously define the pro-$p$ Iwahori of subgroup of ${\rm SL}_n(\mathbb{Z}_p) $ and ${\rm GL}_n(\mathbb{Z}_p) $ which are $p$ saturated if $p>n+1$.

We should distinguish 
the pro-$p$ Iwahori subgroups of ${\rm GL}_n(\mathbb{Z}_p)$ and ${\rm SL}_n(\mathbb{Z}_p)$ 
from their uniform pro-$p$ subgroups, such as their first congruence subgroups.  
Note that the first congruence subgroup 
$\Gamma_1({\rm GL}_n(\mathbb{Z}_p))$ (resp. $\Gamma_1({\rm SL}_n(\mathbb{Z}_p))$) 
of ${\rm GL}_n(\mathbb{Z}_p)$ (resp. ${\rm SL}_n(\mathbb{Z}_p)$) is a uniform pro-$p$ subgroup. 
We should be aware of the fact that an arbitrary 
uniform pro-$p$ group is $p$-valued and $p$-saturated. Although the pro-$p$ Iwahori 
subgroup of ${\rm GL}_n(\mathbb{Z}_p)$ (resp. the pro-$p$ Iwahori subgroup of  ${\rm SL}_n(\mathbb{Z}_p)$) is 
$p$-valued and $p$-saturated as well, it is not in general uniform and does contain the first congruence subgroup 
$\Gamma_1({\rm GL}_n(\mathbb{Z}_p))$ (resp. $\Gamma_1({\rm SL}_n(\mathbb{Z}_p))$) properly. 
It is easily seen that uniform pro-$p$ groups form a subclass of the class of $p$-saturated groups. 
(cf. Remark after \cite[Lemma 4.3]{SchneiderTeitelbaum}). 
Klopsch \cite{Klopsch} and Schneider 
\cite{Schneider} illustrated by examples that the standard notion of uniform pro-$p$ groups 
is more restrictive and less flexible than Lazard's concept of $p$-saturated groups. Klopsch \cite[Proposition 2.4]{Klopsch} also pointed out that the Sylow pro-$p$ subgroups of many classical groups are $p$-saturated, 
but typically fail to be uniform powerful.  Consequently, the Iwasawa algebras 
of the pro-$p$ Iwahori subgroups are much more larger than those of  the first congruence subgroups. 
We therefore say that the Iwasawa algebras over the first congruence subgroups can be looked on as 
subalgebras of the Iwasawa algebras over the pro-$p$ Iwahori subgroups.  

In our future work we hope to generalize our methods  to more general  $p$-saturated groups like the pro-$p$ Iwahori (which is a non-uniform group) of ${\rm GL}_n(\mathbb{Z}_p)$ and ${\rm SL}_n(\mathbb{Z}_p)$ in order to determine the normal elements in their mod-$p$ Iwasawa algebras.  We now need to look at the lowest degree commutators from Ray's article \cite{Ray3} and construct the partial differential equations similar to section \ref{xxsec4}.

\vspace{4mm}

\noindent {\bf Acknowledgements} We have accumulated quite a debt of gratitude in writing this paper. 
Most of all to Professor Konstantin Ardakov and Professor Simon Wadsley for many invaluable suggestions 
and discussions and for constant inspiration. The second author would like to thank PIMS-CNRS and the 
University of British Columbia for postdoctoral research grant. He is also thankful to Beijing Institute of Technology 
for its gracious hospitality during a visit on August 2018 when this collaboration took place.

\end{document}